\documentclass[12pt,twoside,reqno]{amsart}
\usepackage{amsmath,amsfonts,amsthm,amssymb}
\usepackage{graphicx,overpic}
\usepackage{bm}
\usepackage[all,cmtip,line]{xy}
\usepackage{array,setspace}
\usepackage{color}
\numberwithin{equation}{section}

\usepackage{ulem}
\usepackage{graphics}
\usepackage{epsfig}
\newtheorem{claim}{\bf \t}[part]

\newtheorem{theorem}{Theorem}[section]

\newtheorem{lemma}{Lemma}[section]
\newtheorem{proposition}[theorem]{Proposition}
\newtheorem{remark}{Remark}[section]
\newtheorem{definition}{Definition}[section]

\def\t{\theta}

\newcommand \R{\mathbb{R}}

\begin{document}

\title[Low Mach Number Limit of Multidimensional Steady Flows on the Airfoil Problem]
{Low Mach Number Limit of Multidimensional Steady Flows on the Airfoil Problem}

\author{Mingjie Li}
\address{M.J. Li, College of Science, Minzu University of China, Beijing 100081, P. R. China}
\email{lmjmath@163.com}

	\author{Tian-Yi Wang}
	\address{T.-Y. Wang, Department of Mathematics, School of Science, Wuhan University of Technology,
		Wuhan, Hubei 430070, P. R. China;
		Gran Sasso Science Institute, viale Francesco Crispi, 7, 67100 L'Aquila, Italy}
	\email{tianyiwang@whut.edu.cn; tian-yi.wang@gssi.infn.it; wangtianyi@amss.ac.cn}
	
	\author{Wei Xiang}
	\address{W. Xiang, City University of Hong Kong, Kowloon Tong, Hong Kong, P. R. China}
	\email{weixiang@cityu.edu.hk}

\date{\today}

\begin{abstract}
In this paper, we justify the low Mach number limit of the steady irrotational Euler flows for the airfoil problem, which is the first result for the low Mach number limit of the steady Euler flows in an exterior domain. The uniform estimates on  the compressibility parameter $\varepsilon$, which is singular for the flows,  are established via a variational approach based on the compressible-incompressible difference functions. The limit is  on the H\"{o}lder space and is unique. %, which means the better uniform estimates. And 
Moreover, the convergence rate is of order $\varepsilon^2$. It is noticeable that, due to the feature of the airfoil problem, the extra force dominates the asymptotic decay rate of the compressible flow to the infinity. And the effect of extra force vanishes in the limiting process from compressible flows to the incompressible ones, as the Mach number goes to zero. %for other equations.
\end{abstract}


\maketitle

%%%%%%%%%%%%%%%%%%%%%%%%%%%%%%%%%%%%%%%%%%%%%%%%%%%%%%%%%%%%%%%%%%%%%%%%%%%%%%%%%%%%%%%%%%%%%%%%%%

\section{Introduction}

We are concerned with the low Mach number limit of steady homentropic Euler flows to the airfoil problem. The steady homentropic Euler equations with the extra force are written as:
\begin{equation}\label{CISE}
\begin{cases}
\mbox{div}\,(\rho u)=0,\\
\mbox{div}\,(\rho u\otimes u)+ \nabla p =\rho F,
\end{cases}
\end{equation}
where $x=(x_1, \cdots, x_n)\in \R^n, n\ge 3$, $u=(u_1,\cdots,u_n)$ is the velocity,  $\rho$, $p$ and $F$ represent the density, pressure, and extra forces respectively.
Moreover, $u\otimes u=(u_i u_j)_{n\times n}$ is an $n\times n$ matrix. Through this paper, we consider that the extra force $F$ is conservative. 
This is reasonable since this type of forces is quite natural and important in the reality. For instance, the gravity field is a conservative field. Another important example is the electric field. 
%Due the infinity state and the structure of the extra forces,
As the homentropic flow, the pressure is a function of density as:
\begin{equation}
p:=\frac{\tilde{p}(\rho)-\tilde{p}(1)}{\varepsilon^2},
\end{equation}
where $\varepsilon>0$ is the compressibility parameter as introduced in \cite{Schochet}. As usual, we require
\begin{equation}\label{conditiononpressure}
\tilde{p}'(\rho)>0,  \quad 2\tilde{p}'(\rho) + \rho \tilde{p}''(\rho) > 0 \qquad \mbox{for $\rho>0$}.
\end{equation}

We remark that condition \eqref{conditiononpressure} holds for the flows governed by the thermodynamic relation that $\tilde{p}= \rho^\gamma$ with $\gamma\geq 1$.
The sound speed of the flow is $$c:=\sqrt{p'(\rho)}=\frac{\sqrt{\tilde{p}'(\rho)}}{\varepsilon},$$
and the Mach number is  defined as $$M:=\frac{|u|}{c}=\frac{\varepsilon |u|}{\sqrt{\tilde{p}'(\rho)}}.$$ 
%where $q:=|u|=\Big(\sum_{i=1}^n u_i^2\Big)^{1/2}$ is speed. 

Formally, if $|u|$ is bounded and $\sqrt{\tilde{p}'(\rho)}$ does not vanish, the Mach number will go to $0$ as $\varepsilon\rightarrow0$. For this reason, the limit $\varepsilon\rightarrow0$ is called the low Mach number limit.

The expected corresponding homogeneous incompressible Euler equations as $\varepsilon\rightarrow0$ are written as:
\begin{equation}\label{CHE}
\begin{cases}
\mbox{div}\, \bar u=0,\\
\mbox{div}\,(\bar u\otimes \bar u)+ \nabla \bar p=F,
\end{cases}
\end{equation}
where $\bar u=(\bar u_1,\cdots,\bar u_n)$ and $\bar p$ represents the velocity and pressure, respectively, while density $\bar{\rho}\equiv1$.

To study the low Mach number limit, we should start from the case with sufficiently small Mach number. 
The flow is subsonic when the Mach number $M < 1$. 
Since the equations corresponding to the mixed-type characteristics in the compressible Euler equations are elliptic for the subsonic flows, 
we can expect nice regularity %extra-smoothness 
compared to those related to the transonic flows or supersonic flows.
 
The airfoil problems studies the flow past through an exterior domain with the slip boundary condition, which equals to consider the flows around uniform motion body like the airfoil or car after the Galilean transformation. Both the compressible case and incompressible case have been studied by many authors. The first result is due to Frankl and Keldysh \cite{Frankl}, who studied the subsonic flows around a two dimensional airfoil and proved the existence and uniqueness for small data by the method of successive approximations. %Later on, Bers \cite{Bers1} proved the existence of subsonic flows with arbitrarily high local subsonic speed for the Chaplygin gas. 
By the variational method, Shiffman \cite{Shiffman2} proved that, if the speed of the flow at the infinity, $u_{\infty}$, is less than some critical speed, then there exists a unique subsonic potential flow around a given profile with finite energy. Shortly afterwards, Bers \cite{Bers3} improved the uniqueness results of Shiffman. Finn and Gilbarg \cite{Finn1} proved the uniqueness of the two dimensional potential subsonic flow about a bounded obstacle with given circulation and velocity at the infinity. 
All the above results are related to the two dimensional case. For the three (or higher) dimensional case, Finn and Gilbarg \cite{Gilbarg1} proved the existence, uniqueness and the asymptotic behavior with implicit restrictions on the Mach number. Payne and Weinberger \cite{Payne} improved
the results soon after. 
Later, Dong and Ou  \cite{Dong2} extended the results of Finn and Gilbarg
\cite{Gilbarg1} to any Mach number $M < 1$ and to arbitrary dimensions by
the direct method of the calculus of variations and the standard Hilbert space method. Recently, the result has been extended to the case with conservative force in \cite{Gu-Wang2}. 
%Furthermore, in \cite{Dong2}, Dong and Ou extended the results of Shiffman to higher dimensions by
%the direct method of the calculus of variations and the standard Hilbert space method. 
The corresponding incompressible case is considered by Ou in \cite{ou2, ou1}. 
On the other hand, recently, there are many literitures on the steady compressible Euler equations (see \cite{CHWX1,DWX,FX,Gu-Wang1,QX,XZZ} for examples). So it is time for us to consider the low Mach number limit of steady flows.

It is well-known in physics that %From the physical point of view, 
the compressible flow is expected to perform like an %close to 
incompressible flow, when the Mach number is sufficiently small. However, how to rigorously justify the physical observation of this limit is
a challenging mathematical problem, since it is a singular limit and singular phenomena are expected in the limit process.
The first theory of the low Mach number limit is due to Janzen and Rayleigh (see \cite[Sect. 47]{Schiffer}, \cite{VD}), %for expositions
%and references) 
which concerned with the steady irrotational flow. Their method of the expansion of solutions in power with respect to the Mach number was used both as a computational tool and as a method for the proof of the existence of solutions of the compressible flow.  
Klainerman and Majda \cite{KM1, KM2} proved the convergence of compressible to incompressible flow by directly
obtaining estimates for the scaled form of the partial differential equations (also see Ebin \cite{Edin}). In particular, under the suitable initial data, they showed the convergence rates of both Euler equations and Navier-Stokes equations are $\varepsilon$ order on velocity and $\varepsilon^2$ on density and pressure.
By using the fast decay of acoustic waves, Ukai in \cite{Ukai} verified the low Mach number limit for the general  data, while the exterior domain cases were considered in \cite{ Isozaki1}. 
 %In \cite{Schochet1}, Schochet showed the limit of the full compressible Euler equations to the incompressible inhomogeneous Euler equations in bounded domain for local smooth solutions with well-prepared initial data. 
% He \cite{Schochet}  further studied the fast singular limit for general hyperbolic partial differential equations.  
The major breakthrough on the general initial data  is due to M\'{e}tivier and Schochet \cite{MS}, which is extended to the exterior domain problem by Alazard \cite{alazard1}.
 For the one dimensional case, under the $BV$ space, the low Mach number limit have been considered \cite{CCZ}. 
   For the steady Euler flow, Li-Wang-Xiang \cite{Li-Wang-Xiang} considered the infinitely long nozzle problem, which is the first rigorous analysis for the steady Euler equation. For other fluid models, see \cite{Feireisl-N, JJL3, JJLX, Lions P.-L.01, Masmoudi, MS} %,Schochet} 
 and the references therein. For another type of incompressible limit, please see \cite{CHWX, Masmoudi2} and reference therein.

 %  For the isentropic Navier-Stokes equations, the low Mach limit of the global weak solutions with general initial data have been well studied under  various boundary conditions, see \cite{ DeG, DGLM,Lions P.-L.01, Lions-Masmoudi, Masmoudi}. It is noticeable that \cite{Lions P.-L.01}
%  also showed the incompressible limits on the stationary Navier-Stokes equations with the Dirichlet boundary condition, while the gradient estimate on velocity played the major role. 
%For the non-isentropic  Navier-Stokes equations, Alazard \cite{alazard2} justified  the low Mach limit in the whole space for the ill-prepared data, by employing a uniform estimate and the convergence lemma of \cite{MS}.  
%For the bounded domain, the low Mach limit was justified by Jiang-Ou \cite{Jiang-Ou} and Dou-Jiang-Ou \cite{Dou-Jiang-Ou}.  
%A dispersive Navier-Stokes system was also studied in \cite{Levermore}.  
%For other interesting works, see  \cite{DBDL,Danchin,Feireisl-N,Kim,Masmoudi1,Schochet} for Naiver-Stokes equations, 
%\cite{Fan-Gao-Guo,Hu-Wang,JJL1,JJL2,JJL3, JJLX} for  MHD equations and the references therein.

 All the results mentioned above are related to the case that the boundary of the profile is smooth. For the case that the profile is a polygon, recently, Elling \cite{ve} showed that  when the Mach number is sufficiently small but nonzero, the classical solution of the irrotational steady Euler flows  around a polygon does not exist. %, due the regularity loss close to the boundary. \textcolor{red}{We may write more about this?}

%\textcolor{red}{We can delete this part. It is not such necessary to give the details. }For the other case of the incompressible limit, \emph{i.e.}, corresponding to the limit $\gamma\rightarrow \infty$, Lions-Masmoudi \cite{Lions-Masmoudi} shows that the compressible homentropic Navier-Stokes flow will converge to the homogeneous incompressible Navier-Stokes flow. Recently, the compactness framework of this limit for the steady Euler flow is considered by \cite{CHWX}, in which the authors also verified the limits for several cases as the application of the compactness framework.

%To study the low Mach number limit, we should start from the study of the subsonic flow with sufficiently small Mach number, which means the Mach number $M < 1$. 

In this paper, we prove the limit is convergent strongly in the H\"older norm and is unique. Moreover, the convergence rate of order $\varepsilon^2$ is justified, which is higher than the one in Klainerman-Majda \cite{KM1}, due to the irrotational property.
More precisely, %For the low Mach limits, 
the main difficulties for the steady Euler flow are how to get the uniform estimates on the velocity. Therefore, we introduce the variational functional with respect to the compressible-incompressible difference function to obtain the uniform $L^2$ estimates. Moreover, not only are the incompressible and compressible velocities the minimizer of the respective functional but also the compressible-incompressible difference velocity is the minimizer of a functional. The effect of the extra force comes out in the regularity lifting to the H\"older space naturally. In addtion, the effect dominates the asymptotic rate to the infinity for the compressible flows, but the incompressible flows are free from this effect. The vanishing phenomenon is rigorously justified by the proper expansion based on a suitable formulation on the Bernoulli's law and a refined cut-off function for the low Mach number limits. Another difficulty is to find a proper way to show the convergence of the pressure. Based on the observation that the pressure of the incompressible flow is well-defined up to a constant by the Bernoulli's law, we establish the convergence on the gradient of pressure.

The rest of this paper is organized as follows.
In Section 2, we formulate the low Mach number limit of the steady irrotational Euler flows corresponding to the airfoil problem mathematically, and then introduce the main theorem of this paper.
In Sections 3, we derive a variational formulation to show  %direct method of calculus of variations to prove 
the existence and uniqueness of a minimizer of the variational problem. In Section 4, we show the minimizer is actually the unique solution of the compressible subsonic Euler flow in $\R^n$ of the airfoil problem with  $C^{\alpha}$-regularity. %of the solution to the  compressible subsonic Euler flow is derived in Section 4.
%In Sections 5, 
Finally, the convergence rate of the low Mach number limit of the steady irrotational Euler flows is established.

\section{Airfoil Problem and the Low Mach number Limit}
In this section, we will first introduce both the incompressible and compressible airfoil problem, then the low Mach number limit, and finally the main theorem of this paper. 

\subsection{Airfoil problem}
Let $\mathcal{D}(\Gamma)$ (airfoil) be a bounded and connected domain in $\R ^n$ $(n\geq3)$ such that its boundary $\Gamma$ consists of one or several closed and isolated $n-1$ dimensional $C^{2,\alpha}$ (for some $0\leq\alpha\leq1$) hypersurfaces. Let $\Omega$ be the exterior domain
of $\mathcal{D}(\Gamma)$, \textit{i.e.}, $\Omega:=\R^n\backslash\mathcal{D}$, which is connected (see Fig \ref{fig:3dairfoillm}). Then both the incompressible and compressible airfoil problem can be formulated as the following problem. 

\begin{figure}
	\centering
	\includegraphics[width=0.5\linewidth]{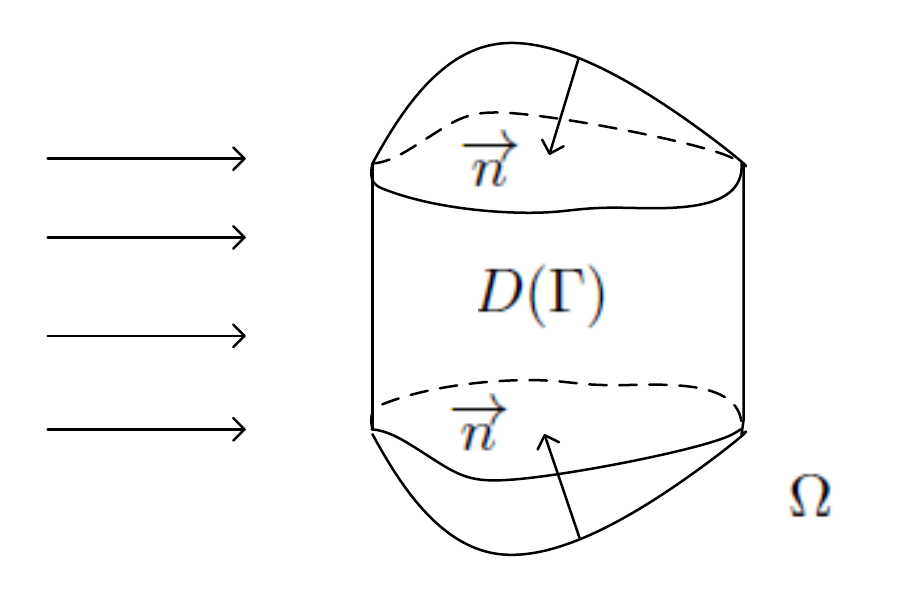}
	\caption{Airfoil Problem}
	\label{fig:3dairfoillm}
\end{figure}

\textbf{Problem 1}.  Let $n\geq3$. Find functions $(\rho,  u, p)$ satisfy \eqref{CISE} or functions $(u,p)$ satisfy
\eqref{CHE}
with the slip boundary
condition
\begin{equation}\label{boundarycondition}
 u\cdot \vec{n}=0\ \ \mbox{on}\ \ \Gamma,
\end{equation}
where $\vec{n}=(n_1,\cdots,n_n)$ denotes the unit inward normal of domain
$\mathcal{D}(\Gamma)$. Moreover, the limits
\begin{equation}\label{u-inftycondtion}
\lim_{|x|\rightarrow \infty} u(x)=u_\infty,
\end{equation}
and
\begin{equation}\label{rhoinftycondtion}
\lim_{|x|\rightarrow \infty} \rho(x)=1,
\end{equation}
exist and are finite.

\begin{remark}
By the Galilean invariance of the Newton fluid, $\mathcal{D}$ can be assumed to be stationary. That is the reason that we assign boundary condition \eqref{boundarycondition} on $\Gamma$. Without loss of the generality, we also assume that $u_{\infty}=(q_{\infty},0, \cdots,0)$. % the constant moving of $\Gamma$ with velocity $u_{\infty}$ could be reduced to the fluid flow around a fixed position of $\Gamma$ with the same velocity at infinity. By proper rotating transition, we could assume $u_{\infty}=(q_{\infty},0, \cdots,0)$ without loss of generality. 
\end{remark}

Since the extra force is conservative form, and velocity is asymptotic irrotational at infinite, we assume the flow is irrotational, which means the vorticity of the flow velocity 
is zero, {\it i.e.}, 
\begin{equation}
\mbox{curl} u=0.
\end{equation}

\subsection{Incompressible airfoil problem}
Now let us consider the existence of the incompressible flow for the airfoil problem.
Due to \cite{Dong2,ou2,ou1}, the suitable function space is of the following. 

\begin{definition}\label{hilb}
$\mathcal{V}$ is the Hilbert space for $n\geq3$, which is the completing of space $\mathcal{V}_0$ under the norm	
\begin{equation*}
\|v\|_{\mathcal{V}} 
= \left(\int_\Omega |\nabla v|^2dx\right)^{\frac{1}{2}}.
\end{equation*}
Here
$\mathcal{V}_0$ is the set of all the functions which are the restrictions of the $C_{c}^{\infty}(\mathbb{R}^n)$ functions on $\Omega$:
	\begin{equation*}
	\mathcal{V}_0 = \left\{v(x) ~ |~  v(x)=V(x) ,~ x\in\Omega, \mbox{~for ~some} ~V\in C_c^{\infty} (\R^n)
	\right\}.
	\end{equation*}
%	Then under the norm 
%\begin{equation*}
%\|v\|_{\mathcal{V}} 
%= \left(\int_\Omega |\nabla v|^2dx\right)^{\frac{1}{2}},
%\end{equation*}
%	$\mathcal{V}_0$ expands a Hilbert space $\mathcal{V}$ if $n\geq 3$.	
\end{definition}

We remark that for the $C_c^{\infty}$ functions, $L^2$ norm of the gradient can be bounded by the norm $\|\cdot\|_{\mathcal{V}}$ due to
the following Hardy type inequality. 
\begin{theorem}\label{Hardy}
	There exists a constant $c(n,  \Omega)$ such that for any $v\in\mathcal{V}$
	\begin{equation}\label{Hardyforv}
	\int_{\Omega}\frac{v^2}{1+|x|^2} dx\leq c(n,  \Omega)\int_{\Omega}|\nabla v|^2 dx, \qquad\mbox{for} \qquad n\geq 3.
	\end{equation}
\end{theorem}

%\begin{figure}
%	\centering
%	\includegraphics[width=10cm]{3dairfoil}
%	\caption{Air Foil}
%	\label{Fig2}
%\end{figure}

Note that the steady irrotational incompressible Euler flow is governed by the following equations: 
\begin{equation}\label{CHEB}
\begin{cases}
\mbox{div}\, \bar{u}=0,\\
\mbox{div}\,(\bar{u}\otimes \bar{u})+ \nabla \bar{p}=F,\\
\mbox{curl}\, \bar{u}=0.
\end{cases}
\end{equation}

Here,  the density $\bar{\rho}\equiv 1$, and conservative force $F$ could be written as $F=\nabla\phi$. By $\eqref{CHEB}_3$, % From the irrotational condition $\mbox{curl}~ \bar{u}=0$, we could 
we can introduce the velocity potential $\bar{\varphi}$ such that %as least locally, which satisfied:
\begin{equation}
\bar{u}=\nabla \bar{\varphi}.
\end{equation}

Therefore, the incompressible irrotational Euler flows of \textbf{Problem 1} can be found by looking for solution  $\bar{\varphi}$, which satisfies 
\begin{equation}\label{li1}
\begin{cases}
\Delta \bar{\varphi} =0, \quad &x\in\Omega,\\
\frac{\partial\bar{\varphi}}{\partial \textbf{n}}=0, \quad &x\in\Gamma,\\
\lim\limits_{|x| \rightarrow \infty}\nabla \bar{\varphi} =(q_\infty, 0, \cdots, 0),
\end{cases}
\end{equation}

Due to the asymptotic behavior of $\nabla\bar{\varphi}$ at infinity (see $\eqref{li1}_3$), we can not expect the $L^2$ estimate of $\nabla\bar{\varphi}$. So we introduce $\bar{\psi}$, such that
$$
\bar{\varphi}=\bar{\psi}+q_\infty x_1.
$$
We will use the variational approach, which is also successfully applied to the incompressible irrotational R\'{e}thy flow problem recently (see \cite{CDX}), to show the existence of solutions of \eqref{li1}.
More precisely, let us consider the following variational problem (see \cite{ou2,ou1})
$$
\min_{\bar{\psi}\in \mathcal{V} %W^{1,2}(\Omega)
}\Big(\int_\Omega \frac{1}{2}|\nabla\bar{\psi}|^2 dx
-q_\infty \int_\Gamma n_1 \bar{\psi} dS\Big).
$$
In \cite{ou1}, Ou proved that this minimization problem admits classical solutions and the solution is unique up to an added constant. Indeed, if $\bar{\psi}$ is a solution of the variational problem, then for any $\eta \in C_c^\infty(\R^n)$, we have
$$
0=\int_\Omega \nabla\bar{\psi}\nabla\eta dx + q_\infty \int_\Gamma n_1 \eta dS
=-\int_\Omega \Delta \bar{\psi} \eta dx +\int_\Gamma 
\left(\frac{\partial \bar{\psi}}{\partial \textbf{n}}+q_\infty n_1\right)\eta dS.
$$
Thus, $\bar{\psi}$ satisfies 
\begin{equation}\label{ell}
\begin{cases}
\displaystyle   \Delta \bar{\psi} =0, \quad &x\in\Omega,\\
\displaystyle   \frac{\partial\bar{\psi}}{\partial \textbf{n}}=
-q_\infty n_1, \quad &x\in\Gamma,
\end{cases}
\end{equation}
in the distributional sense. 
By the standard elliptic theory, $\bar{\psi}$ is smooth up to the boundary and satisfies \eqref{ell} in the classical sense. Moreover, as $x$ approaches to the infinity, $|\nabla \bar\psi|$ tends to zero. Then, the function 
$\bar{\varphi}=\bar{\psi}+q_\infty x_1$ satisfies the equation (\ref{li1}). In summary, we have the  following result.

\begin{theorem}
Problem (\ref{li1}) admits a unique classical solution $\bar{\varphi}=\bar{\psi}+q_\infty x_1$ up to a constant with $\bar\psi\in \mathcal{V}$. Here $\mathcal{V}$ is the Hilbert space defined in Definition \ref{hilb}. Furthermore,
	\begin{equation}\label{barpsiL2}
	\int_{\Omega}|\nabla\bar{\psi}|^2 dx \leq C(q_{\infty}, \Omega),
	\end{equation}
	and
	\begin{equation}\label{barpsiLinfty}
	|\nabla\bar{\psi}|(x) \leq C(q_{\infty}, \Omega) (1+|x|)^{-\frac{n}{2}}.
	\end{equation}
\end{theorem}

Finally, we recall that the Bernoulli law for the incompressible irrotational flow is that
%On the other hand, we could recall the Bernoulli law in the case:
\begin{equation}
\nabla\bar{p}=\nabla\left( \phi-\frac{|\bar{u}^2|}{2}\right).
%\nabla\left( \frac{|\bar{u}^2|}{2}+\bar{p}\right) =0.
\end{equation}
%\textcolor{red}{We may delete this part later.} It leads the pressure satisfies that
%\begin{equation}
%\bar{p}=\frac{(q_\infty)^2-|\bar{u}|^2}{2}=\frac{(q_\infty)^2-|\nabla\bar{\varphi}|^2}{2}
%\end{equation}
%up to a constant. 

\subsection{Compressible airfoil problem and the low Mach number limit}
The irrotational compressible Euler flow with low Mach number can be written as
\begin{equation}\label{CISEV}
\begin{cases}
\mbox{div}\,(\rho^{\varepsilon} u^{\varepsilon})=0,\\
\mbox{div}\,(\rho^{\varepsilon} u^{\varepsilon}\otimes u^{\varepsilon})+ \nabla p^{\varepsilon} =\rho^\varepsilon F,\\
\mbox{curl}\, u^{\varepsilon}=0,
\end{cases}
\end{equation}
with $ p^{\varepsilon}=p^{(\varepsilon)}(\rho^{\varepsilon})=\frac{\tilde{p}(\rho^{\varepsilon})-\tilde{p}(1)}{\varepsilon^2}$. And, conservative force $F$ could be written as $F=\nabla \phi$. The Mach number is defined as $M^\varepsilon=\frac{|u^\varepsilon|}{\sqrt{(p^{(\varepsilon)})'(\rho^\varepsilon)}}=\frac{\varepsilon|u^\varepsilon|}{\sqrt{\tilde{p}'(\rho^\varepsilon)}}$.

By \eqref{boundarycondition}--\eqref{rhoinftycondtion}, the flow satisfies 
the slip boundary
condition that
\begin{equation}\label{bc-1}
u^\varepsilon   \cdot \vec{n}=0\ \ \mbox{on}\ \ \Gamma,
\end{equation}
where $\vec{n}$ denotes the unit inward normal of domain
$\mathcal{D}(\Gamma)$, and the asymptotic behavior as $|x| \rightarrow \infty$ that
\begin{equation}\label{infc-1}
u^{\varepsilon} \rightarrow(q_\infty, 0, \cdots, 0) \qquad \mbox{and}\qquad \rho^{\varepsilon} \rightarrow 1.
\end{equation}

By \eqref{CISEV}, we have the following Bernoulli law that
\begin{equation}
\nabla\left(\frac{|u^\varepsilon|^2}{2}+h^{(\varepsilon)}(\rho^\varepsilon)-\phi \right)=0,
\end{equation}
where the enthalpy $h^{(\varepsilon)}$ is defined by $h^{(\varepsilon)}(\rho)=\int_1^{\rho}\frac{(p^{(\varepsilon)})'(s)}{s} ds$ up to any constant.
%$(h^{(\varepsilon)})'(\rho^\varepsilon)=\frac{(p^{(\varepsilon)})'(\rho^\varepsilon)}{\rho^\varepsilon}$. 

Then by the asymptotic behavior \eqref{infc-1}, we have that
\begin{equation}
\frac{|u^\varepsilon|^2}{2}+h^{(\varepsilon)}(\rho^\varepsilon)%=\frac{q_\infty^2}{2}+h_\varepsilon(\rho_\infty)
=\frac{q_\infty^2}{2}+h^{(\varepsilon)}(1)+\phi,
\end{equation}

Finally, let %$\tilde{h}$ such that 
$$
\tilde{h}=\varepsilon^2 h^{(\varepsilon)}
$$  
with 
$$
\tilde{h}'(\rho)=\frac{\tilde{p}'(\rho)}{\rho}.
$$ 

By \eqref{conditiononpressure}, we can see that $\tilde{h}(\rho)$ is a strictly increasing function with respect to $\rho$, so does $h^{(\varepsilon)}(\rho^\varepsilon)$.  
Let 
$$
\tilde{H}(\rho):=\frac{\tilde{p}'(\rho)}{2}+\tilde{h}(\rho).
$$

Now, we can introduce the critical density %Firstly, let us introduce the $\tilde{H}(\rho):=\frac{\tilde{p}'(\rho)}{2}+\tilde{h}(\rho)$. 
%For  the critical density 
$\rho^\varepsilon_{cr}$ for each fixed $\varepsilon$ such that
\begin{equation}\label{2.17}
\frac{1}{\varepsilon^2}\tilde{H}(\rho^\varepsilon_{cr})=\frac{q_\infty^2}{2}+\frac{1}{\varepsilon^2}\tilde{h}(1)+\phi.
\end{equation}

In this paper, we assume $\phi\in L^{\infty}$. So there exists a constant $\phi^*>0$ such that $\phi$ is bounded as 
\begin{equation}\label{2.18}
|\phi|\leq\phi^\star.
\end{equation}
Note that 
$$
\lim_{\rho\rightarrow+\infty}\tilde{H}(\rho)-\tilde{h}(1)>0>\lim_{\rho\rightarrow0+}\tilde{H}(\rho)-\tilde{h}(1).
$$ 
Then, there exists an $\varepsilon_0>0$, such that when $\varepsilon_0\geq\varepsilon\geq0$, it holds \begin{equation}\label{psicondition}
\frac{1}{\varepsilon^2}\left(\lim_{\rho\rightarrow+\infty}\tilde{H}(\rho)-\tilde{h}(1)\right)-\frac{q_\infty^2}{2}>\phi^*>\phi>-\phi^*>\frac{1}{\varepsilon^2}\left(\lim_{\rho\rightarrow0^+}\tilde{H}(\rho)-\tilde{h}(1)\right)-\frac{q_\infty^2}{2}.\end{equation}

So \eqref{2.17} is solvable, and
\begin{equation}\label{2.19}
\rho^\varepsilon_{cr}(\phi)=\tilde{H}^{-1}\left(\tilde{h}(1)+\frac{\varepsilon^2 q_\infty^2}{2}+\varepsilon^2\phi\right)
\end{equation}
and the critical speed is 
\begin{equation}
q^\varepsilon_{cr}(\phi)=\frac{1}{\varepsilon}\sqrt{\tilde{p}'\circ\tilde{H}^{-1}\left(\tilde{h}(1)+\frac{\varepsilon^2 q_\infty^2}{2}+\varepsilon^2\phi\right)}.
\end{equation}

It is easy to see that $\rho^\varepsilon_{cr}(\phi)\rightarrow\tilde{H}^{-1}(\tilde{h}(1))$ and $\varepsilon q^\varepsilon_{cr}(\phi) \rightarrow \sqrt{\tilde{p}'\circ\tilde{H}^{-1}(\tilde{h}(1))}$, so $q^\varepsilon_{cr}(\phi)$ will go to the infinity as $\varepsilon$ goes to zero. 
%where $h_\varepsilon(\rho)=\frac{1}{\varepsilon^2}\int_1^\rho\frac{\tilde{p}'(s)}{s} ds$.

It is easy to see that $|u^\varepsilon|< q^\varepsilon_{cr}(\phi)$ holds if and only if the flow is subsonic, \textit{i.e.}, $M^\varepsilon(\phi) < 1$. Similarly,  %form the conditions on the pressure \eqref{conditiononpressure}, 
for each $\theta\in(0,1)$, there exists $q^\varepsilon_{\theta}(\phi)$ such that $|u^\varepsilon|\leq q^\varepsilon_{\theta}(\phi)$ holds if and only if $M^\varepsilon(\phi)\leq \theta$. Moreover, $q^{\varepsilon}_{\theta}(\phi)$ is monotonically increasing with respect to $\theta\in(0,1)$. %Similar to $q^\varepsilon_{cr}(\phi)$, for
For fixed $\theta\in(0,1)$, $q^{\varepsilon}_{\theta}(\phi)\rightarrow\infty$ as $\varepsilon\rightarrow 0$. Also, for $\varepsilon>0$, both $\varepsilon q^\varepsilon_{cr}(\phi)$ and  $\varepsilon q^\varepsilon_{\theta}(\phi)$ are uniformly bounded respect to $\varepsilon$. Moreover, when $M^{\varepsilon}(\phi)<1$, density $\rho^{\varepsilon}$ can be represented as the function of   $|u^\varepsilon|^2$ and $\phi$, \textit{i.e.},
\begin{eqnarray}
\rho^\varepsilon=\rho^\varepsilon(|u^\varepsilon|^2,  \phi)&=&\left(h^{(\varepsilon)}\right)^{-1}\left(\frac{q^2_\infty-|u^\varepsilon|^2}{2}+h^{(\varepsilon)}(1)+\phi\right)\nonumber\\
&=&\tilde{h}^{-1}\left(\frac{\varepsilon^2\left(q^2_\infty-|u^\varepsilon|^2\right)}{2}+\tilde{h}(1)+\varepsilon^2\phi\right).\label{2.21}
\end{eqnarray}

 By \eqref{2.19}, for a given fixed $0<\varepsilon<1$, the flow is subsonic if and only if it holds that
\begin{equation}
\tilde{H}^{-1}\left(\tilde{h}(1)+\frac{\varepsilon^2 q_\infty^2}{2}+\varepsilon^2\phi\right)=\rho^\varepsilon_{cr} <\rho^\varepsilon(|u^\varepsilon|^2, \phi).
\end{equation}
Moreover, by \eqref{2.21},
\begin{equation*}
\tilde{H}^{-1}\left(\tilde{h}(1)-\phi^\star\right)<\rho^\varepsilon(|u^\varepsilon|^2, \phi)\leq\tilde{h}^{-1}\left(\frac{\varepsilon^2 q^2_\infty}{2}+\tilde{h}(1)+\varepsilon^2\phi\right)<\tilde{h}^{-1}\left(\frac{ q^2_\infty}{2}+\tilde{h}(1)+\phi^\star\right).
\end{equation*}
%which is uniformly bounded respect to $\varepsilon$.

%\subsection{Main Theorem}
The low Mach number limit is the limit process when $\varepsilon\rightarrow0$, that we expect the compressible Euler flow will converges to the corresponding incompressible Euler flow.
More precisely, we have the following result which is the main theorem of the paper.

\begin{theorem}\label{MainT}
	For a given $L^{\infty}$ function $\phi$ which satisfies
	\begin{equation}\label{psicondition1}
	\phi\in  L^{2}(\Omega)~ ~ \mbox{ and } ~ ~ (1+|x|^\beta)\nabla\phi\in L^q(\Omega) ~ \mbox{ for } ~ q>n , ~~ \beta>1-\frac{n}{q},
	\end{equation}
and for any fixed number $q_{\infty}$, there exists constants $\varepsilon_c>0$ and $\alpha\in(0,1)$, such that when $0<\varepsilon<\varepsilon_c$ there exists a unique solution $(\rho^\varepsilon, u^\varepsilon, p^\varepsilon)\in (C^{\alpha}(\Omega))^{n+2}$ of \textbf{Problem 1} corresponding to equations \eqref{CISEV} with $M^{\varepsilon}<1$. $M^{\varepsilon}$ varies on $(0,1)$ as $\varepsilon$ varies on $(0,\varepsilon_{c})$.  For any $0<\varepsilon<\varepsilon_c$, compressible-incompressible difference velocity $\tilde{u}^{(\varepsilon)}\in (C^{\alpha}(\Omega))^{n}$ satisfies:
$$
\left|\tilde{u}^{(\varepsilon)}\right|\leq
\frac{C}{(1+|x|)^{\beta'}},
$$
where $\beta'=\min\{\frac{n}{2}, \beta+\frac{n}{q}-1\}$. Furthermore, we have that, as $\varepsilon\rightarrow0$ 
\begin{equation*}
\rho^{(\varepsilon)}=1+O(\varepsilon^2)
\qquad
u^{(\varepsilon)}=\bar{u}+\varepsilon^2\tilde{u}^{(\varepsilon)}
\qquad
\nabla p^{(\varepsilon)}=\nabla\bar{p}+O(\varepsilon^2),
\qquad
M^{\varepsilon}=O(\varepsilon)
\end{equation*}
where $(1, \bar{u}, \bar{p})$ is the classical solution of {\bf Problem 1} corresponding to equations \eqref{CHEB}.
\end{theorem}

\begin{remark}
If $n=3$, it is easy to check the force potential generated by the solid domain $\Omega^c$ (which is the complement of the fluid domain $\Omega$):
	$$
	\phi(x)=\int_{\Omega^c}\frac{\rho_s(y)}{|x-y|} d y
	$$
satisfies the conditions \eqref{2.18} and \eqref{psicondition1} for $\phi$, where $x\in\Omega$ and $\rho_s\in L^1(\Omega^c)$ means the density distribution in $\Omega^c$ is of finite mass. As an example, $\phi$ can be the electric field.
\end{remark}

\begin{remark}
By \eqref{psicondition1} and the Gagliardo-Nirenberg interpolation inequality, we know that $\phi\in W^{1,q}(\Omega)$ if $n\geq4$ or $q>6$ if $n=3$. So $\phi$ is an $L^{\infty}$ function by the Morrey's  inequality. Hence condition \eqref{2.18}, and then condition \eqref{psicondition} follow.
\end{remark}

\begin{remark}
In Theorem \ref{MainT}, the regularity of $(\rho^{(\varepsilon)}, u^{(\varepsilon)}, p^{(\varepsilon)})$ are limited by $\phi$. One can improve the regularity by imposing higher regularity.
\end{remark}

\begin{remark}
The effect of conservative force $\nabla\phi$ is very clear here. Since $\bar{u}-q_\infty e_1$ have the $\frac{n}{2}$ decay rate to infinity, the first order of compressible part $\tilde{u}^{(\varepsilon)}$ may have slower convergence rate due to the decay of $\nabla\phi$.
\end{remark}

\subsection{Potential formulation of compressible flow }

Similar to the incompressible case, we can also introduce the velocity potential $\varphi^{(\varepsilon)}$ for the compressible case that 
\begin{equation}
\nabla \varphi ^{(\varepsilon)}= u^{\varepsilon},
\end{equation}
with the slip boundary
condition that
\begin{equation}\label{bc}
\frac{\partial \varphi^\varepsilon}{\partial  \textbf{n}}=0\ \ \mbox{on}\ \ \Gamma,
\end{equation}
where  \textbf{n} denotes the unit inward normal of domain
$\mathcal{D}(\Gamma)$, and the asymptotic behavior at the infinity that
\begin{equation}\label{infc}
\lim\limits_{|x| \rightarrow \infty}\nabla {\varphi}^{\varepsilon} =(q_\infty, 0, \cdots, 0).
\end{equation}

Therefore, {\bf Problem 1} can be reformulated into the following problem. 

\textbf{Problem 2}. Let $n\geq3$. Find function $\varphi^{(\varepsilon)}$ such that
\begin{equation}\label{ce1}
\begin{cases}
\mbox{div} \left(\rho^\varepsilon(|\nabla\varphi^{(\varepsilon)}|^2, \phi)\nabla\varphi^{(\varepsilon)} \right)=0, \quad &x\in\Omega,\\
\frac{\partial\varphi^{(\varepsilon)}}{\partial \textbf{n}}=0, \quad &x\in\Gamma,\\
\lim\limits_{|x| \rightarrow \infty}\nabla \varphi^{(\varepsilon)} =(q_\infty, 0, \cdots, 0).
\end{cases}
\end{equation}

By the straightforward computation, $\eqref{ce1}_1$ can be rewritten as
\begin{equation}\label{comell}
\sum_{ij=1}^n a_{ij}^{\varepsilon}\partial_{ij}\varphi^{(\varepsilon)}+\sum_{i=1}^n b^\varepsilon_i\partial_i\varphi^{(\varepsilon)}=0,
\end{equation}  
where
\begin{equation}
a^\varepsilon_{ij}
%=\rho^{\varepsilon}\left(\delta_{ij}-\frac{\partial_i\varphi^{(\varepsilon)}\partial_j\varphi^{(\varepsilon)}}
%{(c^\varepsilon)^2}\right)
=\rho^{\varepsilon}\left(\delta_{ij}-\frac{\varepsilon^2\partial_i\varphi^{(\varepsilon)}\partial_j\varphi^{(\varepsilon)}}
{\tilde{p}'(\rho^\varepsilon)}\right),
\end{equation}
and
\begin{equation}
b^\varepsilon_i=\frac{\varepsilon^2\rho^\varepsilon\partial_i\phi}{\tilde{p}'(\rho^\varepsilon)}.
\end{equation}
%From the direct calculation, we have for any $\xi\in\mathbb{R}^n$,
%\begin{equation}
%\rho^\varepsilon\left(1-(M^\varepsilon)^2\right)|\xi|^2\leq\sum_{ij=1}^n a^\varepsilon_{ij}\xi_i\xi_j\leq \rho^\varepsilon|\xi|^2.
%\end{equation}
For $0<\varepsilon<1$ and $M^\varepsilon<\theta<1$, we have that 
\begin{equation}
0<\lambda_1|\xi|^2\leq\sum_{ij=1}^n a^\varepsilon_{ij}\xi_i\xi_j\leq \lambda_2|\xi|^2,
\end{equation}
where constants $\lambda_1$ and $\lambda_2$ do not depend on %are uniformly bounded respect to 
$\varepsilon$.

%\begin{remark}
%	Since we are considering $\varepsilon\rightarrow0$, from now on, we are assuming $0<\varepsilon<1$.  From now on, we just consider the case that $\varepsilon$ is sufficiently small and $\theta$ is close to $1$, such that $q^\varepsilon_\theta>q_\infty$. 
%\end{remark}

\section{Variational Approach of the Compressible airfoil problem}
To show Theorem \ref{MainT}, we should first show the existence of smooth solution of the compressible airfoil problem. To show it, in this section, we will first introduce a variational formulation, and then show the existence of the minimizer which is a solution of the compressible airfoil problem. %, and finally consider the regularity of the obtained solution.
 
\subsection{A variational formulation for the compressible flow}
Notice that equation \eqref{comell} is nonlinear and is uniformly elliptic if and only if $M^\varepsilon<\theta<1$. %However,  a priori estimate $|\nabla\varphi^{(\varepsilon)}|\leq q^\varepsilon_{\theta}$ for some $0<\theta<1$ is missing. So, 
Therefore we need to introduce the cut-off to truncate the coefficients of equation \eqref{comell}. For $0<\varepsilon_0<1$, $0<\theta<1$, we introduce $\mathring{q}^{\varepsilon_0}_\theta(\phi)=\inf_{0<\varepsilon<\varepsilon_0} q^\varepsilon_\theta(\phi)$, and cut-off function on the phase plan
%\begin{equation*}
%\hat{\rho}^\varepsilon(q^2; \phi^\varepsilon)=\begin{cases}
%\rho^\varepsilon(q^2;\phi^\varepsilon)\quad
%&\mbox{if}~|q|\leq \mathring{q}^{\varepsilon_0}_\theta,\\
%\mbox{monotone smooth function} \quad
%&\mbox{if} 
%~\mathring{q}^{\varepsilon_0}_\theta   \leq|q|\leq \mathring{q}^{\varepsilon_0}_{\frac{1+\theta}{2}},\\
%\sup_{x\in\Omega}\rho^\varepsilon\left(\left(\mathring{q}^\varepsilon_{\frac{1+\theta}{2}}\right)^2;\phi^\varepsilon\right)
%\quad
%&\mbox{if}~|q|\geq \mathring{q}^{\varepsilon_0}_{\frac{1+\theta}{2}}(x).
%\end{cases}\end{equation*}
\begin{equation*}
\hat{q}(q^2, \phi)=\begin{cases}
q^2-2\phi\quad
&\mbox{if}~|q|\leq \mathring{q}^{\varepsilon_0}_\theta(\phi),\\
\mbox{monotone smooth function} \quad
&\mbox{if} 
~\mathring{q}^{\varepsilon_0}_\theta(\phi)   \leq|q|\leq \mathring{q}^{\varepsilon_0}_{\frac{\theta+1}{2}}(\phi),\\
\sup_{x\in\Omega}\left(\left(\mathring{q}^{\varepsilon_0}_{\frac{\theta+1}{2}}\right)^2(\phi)-2\phi\right)(x)\quad
&\mbox{if}~|q|\geq \mathring{q}^{\varepsilon_0}_{\frac{\theta+1}{2}}(\phi),
\end{cases}\end{equation*}

Let $\hat{\rho}^{(\varepsilon)}$ satisfy
\begin{equation}\label{3.1}
\frac{\hat{q}(q^2, \phi)}{2}+h^{(\varepsilon)}(\hat{\rho}^{(\varepsilon)})%=\frac{q_\infty^2}{2}+h_\varepsilon(\rho_\infty)
=\frac{q_\infty^2}{2}+h^{(\varepsilon)}(1),
\end{equation}
which is equivalent to that
\begin{equation}
\hat{\rho}^\varepsilon=\hat{\rho}^\varepsilon(q^2, \phi)
=\tilde{h}^{-1}\left(\frac{\varepsilon^2\left(q^2_\infty-\hat{q}(q^2, \phi)\right)}{2}+\tilde{h}(1)\right).
\end{equation}

%For any fixed $q_{\infty}$, choose $\varepsilon_c\in(0,1)$ small such that $q^\varepsilon_\theta>q_\infty$ for any $\varepsilon\in(0,\varepsilon_c)$. So we have that for any $q$,  $q_\infty^2-\hat{q}^{(\varepsilon)}(q^2)\leq |q_\infty^2-q^2|$.

We remark that in this paper, we will construct solutions that satisfy 
%From now on, we consider the $\varepsilon_0$ closes to $1$ and $\theta$ closes to $1$ such that 
$|\nabla\bar{\varphi}|\leq \mathring{q}^{\varepsilon_0}_\theta$, which implies $\hat{q}(|\nabla\bar{\varphi}|^2)=|\nabla\bar{\varphi}|^2$. Moreover, define 
$
\mathring{q}^{\varepsilon_0}_{cr}=\inf_{0<\varepsilon<\varepsilon_0} q^\varepsilon_{cr}.%=\inf_{0<\varepsilon<\varepsilon_0} q^\varepsilon_{1}.
$
And, we denote $\hat{\rho}^\varepsilon_\Lambda(\Lambda, \phi):=\frac{\partial}{\partial \Lambda}\hat{\rho}^\varepsilon(\Lambda, \phi)$, and $\hat{\rho}^\varepsilon_\phi(\Lambda, \phi):=\frac{\partial}{\partial \phi}\hat{\rho}^\varepsilon(\Lambda, \phi)$.

Then, \textbf{Problem 2} is reformulated into \textbf{Problem 3} as follows.

\textbf{Problem 3}:
Let $n\geq3$. Find function $\varphi^{(\varepsilon)}$ which satisfies
\begin{equation}\label{equ-modify}
\begin{cases}
\mbox{div} \left(\hat{\rho}^\varepsilon(|\nabla\varphi^{(\varepsilon)}|^2, \phi)\nabla\varphi^{(\varepsilon)}\right)=0, \quad &x\in\Omega,\\
\frac{\partial\varphi^{(\varepsilon)}}{\partial \textbf{n}}=0, \quad &x\in\Gamma,\\
\lim\limits_{|x| \rightarrow \infty}\nabla \varphi^{(\varepsilon)}=(q_\infty, 0, \cdots, 0).
\end{cases}
\end{equation}

Straightforward calculation yields that  $\eqref{equ-modify}_1$ can be rewritten as
$$
\sum_{i,j=1}^n\hat{a}_{ij}(\nabla \varphi^{(\varepsilon)}, \phi)\partial_{ij}\varphi^{(\varepsilon)}+\sum_{i=1}^n\hat{b}_i(\nabla\varphi^{(\varepsilon)}, \phi)\partial_i\varphi^{(\varepsilon)}=0,
$$ 
where
\begin{eqnarray}\label{li-3.2}
\hat{a}_{ij}\left(\nabla \varphi^{(\varepsilon)}, \phi\right) &=& \hat{\rho}^{\varepsilon}\left(|\nabla\varphi^{(\varepsilon)}|^2, \phi\right)\left(\delta_{ij} -
\frac{\hat{q}_\Lambda(|\nabla\varphi^{(\varepsilon)}|^2, \phi)\partial_i\varphi^{(\varepsilon)}\partial_j\varphi^{(\varepsilon)}}
{(c^\varepsilon)^2}\right)\nonumber\\
&=&\hat{\rho}^{\varepsilon}\left(|\nabla\varphi^{(\varepsilon)}|^2\right)\left(\delta_{ij} - 
\frac{\varepsilon^{2}\hat{q}_\Lambda(|\nabla\varphi^{(\varepsilon)}|^2, \phi)\partial_i\varphi^{(\varepsilon)}\partial_j\varphi^{(\varepsilon)}}
{\tilde{p}'(\hat{\rho}^{\varepsilon})}\right),
\end{eqnarray}
and
\begin{equation}
\hat{b}_i\left(\nabla \varphi^{(\varepsilon)}, \phi\right)=\frac{\varepsilon^2\hat{\rho}^\varepsilon\hat{q}_\phi(|\nabla\varphi^{(\varepsilon)}|^2, \phi)\partial_i\phi}{\tilde{p}'(\hat{\rho}^\varepsilon)}
\end{equation}

Moreover, %And there exist two positive constant $\alpha$ and $\beta$ such that
\begin{equation}\label{li-3.3}
%\hat{\rho}^{(\varepsilon)} (s^2) + 2\left(\hat{\rho}^{(\varepsilon)}\right)'(s^2)s^2<C, \qquad
\hat{\lambda}_1 |\xi|^2 \leq\sum_{i, j =1}^n \hat{a}_{ij} \left(\nabla\varphi^{(\varepsilon)}, \phi\right)\xi_i\xi_j\leq \hat{\lambda}_2 |\xi|^2, \qquad |\hat{b}_i \partial_i\varphi^{(\varepsilon)}|\leq C|\partial_{i}\phi|,
\end{equation}
where constants $C$, $\hat{\lambda}_1$, and $\hat{\lambda}_2$ depend only on the subsonic truncation parameters $\theta$ and $\varepsilon_0$, and do not depend on solution $\varphi^{(\varepsilon)}$. 

Finally, we remark that a solution of \textbf{Problem 3}, where the density $\hat{\rho}^{(\varepsilon)}$ is derived from the new density-speed relation \eqref{3.1}, is also a solution of the original potential flow equation in the  \textbf{Problem 2} when $|\nabla\varphi^{(\varepsilon)}|<\mathring{q}^{\varepsilon_0}_\theta$.

To solve \textbf{Problem 3}, we follow the idea used in \cite{Dong2} to introduce a variational formulation. Denote
$$
G^{(\varepsilon)}(\Lambda, \phi) = \frac{1}{2}\int_0^\Lambda \hat{\rho}^{(\varepsilon)}(\lambda, \phi)d\lambda.
$$

Formally, \eqref{equ-modify} is the Euler-Lagrangian equation of the variational problem with respect to the integral that 
$$
\int_\Omega G^{(\varepsilon)}\left(|\nabla\varphi^{(\varepsilon)}|^2, \phi\right)dx.
$$
%$$
%\delta\int G^{(\varepsilon)}\left(|\nabla\varphi^{(\varepsilon)}|^2\right)dx=0.
%$$

However, the integral above is infinite due to the unbounded domain $\Omega$. To overcome it,
%To overcome this divergence integral in the unbounded domain $\Omega$, 
we introduce 
\begin{equation}\label{li-3.1}
I^{(\varepsilon)}\left(\varphi, \bar{\varphi} \right) 
=\varepsilon^{-4}\int_\Omega 
\left[G^{(\varepsilon)}\left(|\nabla\varphi|^2, \phi\right)
-G^{(\varepsilon)}\left(|\nabla\bar{\varphi}|^2, \phi\right)
-\nabla\bar{\varphi}\cdot\left(\nabla\varphi-\nabla\bar{\varphi}\right)
\right]dx.
\end{equation}
%\begin{equation}\label{li-3.1}
%I^{(\varepsilon)}\left(\varphi^{(\ve)}, \bar{\varphi} \right) 
%=\varepsilon^{-4}\int_\Omega 
%\left[G^{(\varepsilon)}\left(|\nabla\varphi^{(\varphi)}|^2\right)
%-G^{(\varepsilon)}\left(|\nabla\bar{\varphi}|^2\right)
%-2\nabla\bar{\varphi}\cdot\left(\nabla\varphi^{(\varepsilon)}-\nabla\bar{\varphi}\right)
%\right]dx.
%\end{equation}

%Obviously, \eqref{li-3.1} is finite and its Euler-Lagrangian equation is \eqref{equ-modify}.

%Let %Next, we always use $\tilde{\varphi}$ as 
%$\tilde{\varphi}=\frac{\varphi-\bar{\varphi}}{\varepsilon^2}$ for a given function $\varphi$. % respect to $\varphi$.
%Then it is easy to verify that
%%It turns out that
%\begin{equation}\label{vacom}	
%\min_{\tilde{\varphi}\in \mathcal{V}} I^{(\varepsilon)}(\varphi, \bar{\varphi})=\min_{\tilde{\varphi}\in \mathcal{V}} I^{(\varepsilon)}(\bar{\varphi}+\varepsilon^{2}\tilde{\varphi}, \bar{\varphi})
%\end{equation}
Then \textbf{Problem 3} becomes: 

\textbf{Problem 4 }: Find a minimizer $\tilde{\varphi}^{(\varepsilon)}\in\mathcal{V}$  such that
\begin{equation}
%I^{(\varepsilon)}\left(\varphi^{(\varepsilon)}, \bar{\varphi}\right)=
I^{(\varepsilon)}\left(\bar{\varphi}+\varepsilon^2\tilde{\varphi}^{(\varepsilon)}, \bar{\varphi}\right)
%=\min_{\tilde{\varphi}\in\mathcal{V}}I^{(\varepsilon)}\left(\varphi, \bar{\varphi}\right)
=\min_{\tilde{\varphi}\in \mathcal{V}} I^{(\varepsilon)}(\bar{\varphi}+\varepsilon^{2}\tilde{\varphi}, \bar{\varphi}).
\end{equation}

Here, $\tilde{\varphi}=\frac{\varphi-\bar{\varphi}}{\varepsilon^2}$ and $\tilde{\varphi}^{(\varepsilon)}=\frac{\varphi^{(\varepsilon)}-\bar{\varphi}}{\varepsilon^2}$.

%Let %$\tilde{\varphi}^{(\varepsilon)}$ as 
%$$
%\tilde{\varphi}^{(\varepsilon)}=\frac{\varphi^{(\varepsilon)}-\bar{\varphi}}{\varepsilon^2}.
%$$ 

\begin{remark}\label{rem:3.1}
The minimizer of \textbf{Problem 4} satisfies \eqref{equ-modify}.
\end{remark}
\begin{proof}
%It is direct to check that the equation \eqref{equ-modify} is the Euler-Lagrangian equation of our variation problem. 
Let
$\eta\in C_c^\infty({\R}^n)$.
The first variation of $I^{(\varepsilon)}\left(\varphi^{(\varepsilon)}, \bar{\varphi}\right)$ with
$\eta$ is
\begin{equation}
0=\int_\Omega\left[2G^{(\varepsilon)}_\Lambda \left(|\nabla\varphi^{(\varepsilon)} |^2, \phi\right)\nabla\varphi^{(\varepsilon)}\cdot\nabla\eta-\nabla\bar{\varphi}\cdot\nabla\eta\right]dx
\end{equation}
The second term in the integrand above vanishes due to the definition of the incompressible flow $\bar{\varphi}$. Then, the first variation of
$I^{(\varepsilon)}\left(\varphi^{(\varepsilon)}, \bar{\varphi}\right)$ 
associated with $\eta$ is
\begin{eqnarray}
0&=&\int_\Omega 2  G^{(\varepsilon)}_\Lambda \left(|\nabla\varphi^{(\varepsilon)}|^2, \phi\right)\nabla\varphi^{(\varepsilon)}\cdot\nabla\eta dx\nonumber\\
&=&\int_\Omega \tilde{\rho}^{(\varepsilon)} \left(|\nabla\varphi^{(\varepsilon)}|^2, \phi\right)\nabla\varphi^{(\varepsilon)}\cdot\nabla\eta dx\nonumber
%\\
%&=&-\int_\Omega\mbox{div} \left(\tilde{\rho}^{(\varepsilon)} \left(|\nabla\varphi^{(\varepsilon)}|^2\right)
%\nabla\varphi^{(\varepsilon)}\right)\eta dx
%+\int_\Gamma \tilde{\rho}^{(\varepsilon)} \left(|\nabla\varphi^{(\varepsilon)}|^2\right)
%\frac{\partial\varphi^{(\varepsilon)}}{\partial n}\eta dS.\nonumber
\end{eqnarray}
Hence this remark follows from the identity above.
\end{proof}

\subsection{Unique existence of the minimizer of Problem 4}
First, for the existence of a minimizer of \textbf{Problem 4}, we have the following theorem. %For our variational problem, we have the following theorem:
\begin{theorem}\label{th3} \textbf{Problem 4} admits a unique minimizer $\tilde{\varphi}^{(\varepsilon)}\in\mathcal{V}$, which satisfies that %. Moreover,
	\begin{equation}
	\label{ff}
	\int_{\Omega}|\nabla\tilde{\varphi}^{(\varepsilon)}|^2\, dx \leq C,
	\end{equation}
	where
	%	\begin{equation}
	%		\varphi^{(\ve)} = \bar{\varphi}+\ve^2\tilde{\varphi}^{(\ve)},
	%	\end{equation}
constant $C$ depends only on $\bar{\varphi}$ and $\Omega$ and does not depend on $\varepsilon$.
\end{theorem}

\begin{proof}%By following the standard calculate of variation method, t
The proof is divided into four steps.

	\textbf{Step 1.} $I^{(\varepsilon)}(	\bar{\varphi}+\varepsilon^{2}\tilde{\varphi}, \bar{\varphi})=I^{(\varepsilon)}(\varphi, \bar{\varphi})$ is coercive with respect to $\tilde{\varphi}$ in $\mathcal{V}$, \textit{i.e.}, we will show that
\begin{equation}\label{I-phi}
	I^{(\varepsilon)}(\varphi, \bar{\varphi}) \geq \frac{C_1}{2}\int_{\Omega} |\nabla\tilde{\varphi}|^2dx - C_3\int_{\Omega} |\nabla\bar{\varphi}-e_1q_\infty|^2dx
-C\int_{\Omega} |\phi|^2dx.
\end{equation}
	
Let
	\begin{eqnarray}
	&&I^{(\varepsilon)}_1(\varphi, \bar{\varphi})\nonumber\\
	:&=&	\varepsilon^{-4}\int_\Omega 
	\left[G^{(\varepsilon)}\left(|\nabla\varphi|^2, \phi\right)
	-G^{(\varepsilon)}\left(|\nabla\bar{\varphi}|^2, \phi\right)
	-2 G^{(\varepsilon)}_\Lambda \left(|\nabla\bar{\varphi} |^2, \phi\right)\nabla\bar{\varphi}\cdot\left(\nabla\varphi-\nabla\bar{\varphi}\right)
	\right]dx,\nonumber
	\end{eqnarray}
	and
	\begin{equation}
	I^{(\varepsilon)}_2(\varphi, \bar{\varphi})%\nonumber\\
	:= \varepsilon^{-4}\int_\Omega 
	\left[
	 2\left(G^{(\varepsilon)}_\Lambda \left(|\nabla\bar{\varphi} |^2, \phi\right)-1\right)\nabla\bar{\varphi}\cdot\left(\nabla\varphi-\nabla\bar{\varphi}\right)
	\right]dx.%\nonumber
	\end{equation}
Obviously,
%	\begin{equation}
$	I^{(\varepsilon)}(\varphi, \bar{\varphi})=I^{(\varepsilon)}_1(\varphi, \bar{\varphi})+I^{(\varepsilon)}_2(\varphi, \bar{\varphi}).
$%	\end{equation}

First, we will show that $I_1^{(\varepsilon)}(\varphi,\bar{\varphi})$ is coercive with respect to $\tilde{\varphi}$ in $\mathcal{V}$.	
	
	We denote $p=(p_1, \cdots, p_n)$, $F^{(\varepsilon)}(p)=G^{(\varepsilon)}\left(|p|^2 , \phi\right)$. Then by direct computation, we can get that
	\begin{eqnarray}
	&&G^{(\varepsilon)}\left(|\nabla\varphi|^2, \phi\right)-G^{(\varepsilon)}\left(|\nabla\bar{\varphi}|^2, \phi\right)
	-2 G^{(\varepsilon)}_\Lambda\left(|\nabla\bar{\varphi}|^2, \phi\right) \nabla\bar{\varphi}\cdot\left(\nabla\varphi-\nabla\bar{\varphi}\right)\nonumber\\
	&=& F^{(\varepsilon)}(\nabla\varphi) - F^{(\varepsilon)}\left( \nabla\bar{\varphi}\right)
	-\nabla F^{(\varepsilon)} \left( \nabla\bar{\varphi}\right) \cdot \left(\nabla\varphi- \nabla\bar{\varphi}\right)\nonumber\\
	&=&\sum_{i, j=1}^n\int_0^1(1-t)\partial_{p_ip_j} F^{(\varepsilon)}\left(t\nabla\varphi+(1-t) \nabla\bar{\varphi}\right)dt
	\partial_i(\varphi-\bar{\varphi})\partial_j(\varphi-\bar{\varphi}).\nonumber
	\end{eqnarray}
	It is easy to check $\partial_{pp}^2F^{(\varepsilon)}$ is uniformly positive.  In fact, we have
	\begin{equation}
	\left(\partial_{pp}^2F^{(\varepsilon)}(p)\right)_{i,j}%=\hat {\rho}^{(\varepsilon)}\delta_{ij}	+ 2 (\hat{\rho}^{(\varepsilon)})' p_ip_j
	= \hat{a}_{ij}\left(\nabla \varphi^{(\varepsilon)}, \phi\right) .
	\end{equation}
	
	From property \eqref{li-3.3}, we get the uniformly positivity of $\partial_{pp}^2F^{(\varepsilon)}$. As a consequence, we have
	%	\begin{equation}\label{ineq-phi}
	%	\frac{C_1}{2}\varepsilon^{-2}\int|\nabla(\varphi-\bar{\varphi})|^2dx\leq I^{(\varepsilon)}_1(\varphi, \bar{\varphi}) \leq
	%	\frac{\tilde{C}_1}{2}\varepsilon^{-2}\int|\nabla(\varphi-\bar{\varphi})|^2dx.
	%	\end{equation}
	\begin{equation}\label{ineq-phi}
	\frac{C_1}{2}\int_{\Omega}|\nabla\tilde{\varphi}|^2dx\leq I^{(\varepsilon)}_1(\varphi, \bar{\varphi}) \leq
	\frac{\tilde{C}_1}{2}\int_{\Omega}|\nabla\tilde{\varphi}|^2dx.
	\end{equation}

Now, let us consider $I^{(\varepsilon)}_2(\varphi, \bar{\varphi})$. Note that 
	\begin{eqnarray}
	\left|\varepsilon^{-2}\left( 2 G^{(\varepsilon)}_\Lambda \left(|\nabla\bar{\varphi} |^2, \phi\right)-1\right)\right|	&=&\left|\varepsilon^{-2}\left( \hat{\rho}^{(\varepsilon)}  \left(|\nabla\bar{\varphi} |^2, \phi\right)-1\right)\right|	\nonumber\\
	&=&\left|\varepsilon^{-2}\left( 	\tilde{h}^{-1}\left(\frac{\varepsilon^2\left(q^2_\infty-\hat{q}(|\nabla\bar{\varphi}|^2, \phi)\right)}{2}+\tilde{h}(1)\right)-1\right)\right|\nonumber\\
	&=&\left|\varepsilon^{-2}\left( 	\tilde{h}^{-1}\left(\frac{\varepsilon^2\left(q^2_\infty-|\nabla\bar{\varphi}|^2+2\phi\right)}{2}+\tilde{h}(1)\right)-\tilde{h}^{-1}\left(\tilde{h}(1)\right)\right)\right|\nonumber\\
	&=& \left|\frac{q^2_\infty-|\nabla\bar{\varphi}|^2+2\phi}{2}\int_{0}^{1} 	\left( \tilde{h}^{-1} \right)'\left(t\frac{\varepsilon^2\left(q^2_\infty-|\nabla\bar{\varphi}|^2+2\phi\right)}{2}+\tilde{h}(1)\right) dt\right|%\nonumber\\
%	&&
\nonumber\\
	&\leq&\left(|\nabla\bar{\varphi}-e_1q_\infty| \left|\frac{q_\infty+|\nabla\bar{\varphi}|}{2}\right|+|\phi|\right)\nonumber\\
	&&\left|\int_{0}^{1} 	\left( \tilde{h}^{-1} \right)'\left(t\frac{\varepsilon^2\left(q^2_\infty-|\nabla\bar{\varphi}|^2+2\phi\right)}{2}+\tilde{h}(1)\right) dt\right|%\nonumber\\
	%&& 
	.\nonumber
	\end{eqnarray}
	
Then	
		\begin{eqnarray}\label{i2}
	|I^{(\varepsilon)}_2(\varphi, \bar{\varphi})|
	&=&\left|	\varepsilon^{-4}\int_\Omega 
	\left[
	\left( 2G^{(\varepsilon)}_\Lambda \left(|\nabla\bar{\varphi} |^2, \phi\right)-1\right)\nabla\bar{\varphi}\cdot\left(\nabla\varphi-\nabla\bar{\varphi}\right)
	\right]dx\right|,\nonumber\\
	&\leq& C \varepsilon^{-2}\int_{\Omega} \left(|\nabla\bar{\varphi}-e_1q_\infty|+|\phi|\right) |\nabla(\varphi-\bar{\varphi})|
	dx\nonumber\\
	&\leq& C\int_{\Omega} |\nabla\bar{\varphi}-e_1q_\infty|^2dx+C\int_{\Omega} |\phi|^2dx+\frac{C_1}{4} \varepsilon^{-4}\int_{\Omega} |\nabla(\varphi-\bar{\varphi})|^2dx.
	\end{eqnarray}	
	
%Now we can prove the coercive for $I^{(\varepsilon)}(\varphi, \bar{\varphi})$. 
By \eqref{ineq-phi}, \eqref{i2}, and the definition of $\tilde{\varphi}$ together, we get
\eqref{I-phi},
which also implies that $I^{(\varepsilon)}(\varphi, \bar{\varphi})$ is bounded from below, \textit{i.e.}, there exists a constant $C>0$ depending only on the data such that
\begin{equation}
I^{(\varepsilon)}(\varphi,\bar{\varphi})\geq -C
\end{equation}
for all $\tilde{\varphi}\in\mathcal{V}$.

\textbf{Step 2}.
Similarly, by \eqref{ineq-phi} and \eqref{i2} together, we can also show the upper bound of $I^{(\varepsilon)}(\varphi, \bar{\varphi})$ as:
\begin{equation}\label{I-phi-1}
I^{(\varepsilon)}(\varphi, \bar{\varphi}) \leq \frac{C_1+2\tilde{C}_1}{2}\int_{\Omega} |\nabla\tilde{\varphi}|^2dx + C_3\int_{\Omega} |\nabla\bar{\varphi}-e_1q_\infty|^2dx+ C_4\int_{\Omega} |\phi|^2dx.
\end{equation}

It means that $I^{(\varepsilon)}(\varphi,\bar{\varphi})=I^{(\varepsilon)}(\bar{\varphi}+\varepsilon^2\tilde{\varphi},\bar{\varphi})$ is finite for any $\tilde{\varphi}\in\mathcal{V}$. 
%Therefore, there exists a $\tilde{\varphi}\in\mathcal{V}$, such that 
%\begin{equation}
%I^{(\varepsilon)}(\bar{\varphi}+\varepsilon^2\tilde{\varphi},\bar{\varphi})<\infty.
%\end{equation}

\textbf{Step 3.} 	We will prove $I^{(\varepsilon)}(\varphi, \bar{\varphi})$ is uniformly convex in the space $\mathcal{V}$.

	Since $I^{(\varepsilon)}_2(\varphi, \bar{\varphi})$ is linear respect to $\tilde{\varphi}$. for any $\tilde{\varphi}_1, \tilde{\varphi}_2\in \mathcal{V}$, we have that
\begin{eqnarray}
&&I^{(\varepsilon)}\left(\varphi_1,  \bar{\varphi}\right) + I^{(\varepsilon)} \left(\varphi_2,  \bar{\varphi}\right)
-2I^{(\varepsilon)}\left(\frac{\varphi_1+\varphi_2}{2},  \bar{\varphi}\right)\nonumber\\
&=&I_1^{(\varepsilon)}\left(\varphi_1,  \bar{\varphi}\right) + I_1^{(\varepsilon)} \left(\varphi_2,  \bar{\varphi}\right)
-2I_1^{(\varepsilon)}\left(\frac{\varphi_1+\varphi_2}{2},  \bar{\varphi}\right)\nonumber\\
&=&\int_{\Omega}F(\nabla\varphi_1) + F(\nabla\varphi_2)-2F\left(\frac{\nabla\varphi_1+\nabla\varphi_2}{2}\right) dx\nonumber\\
&\geq& \frac{C_1}{2} \varepsilon^{-4}\|\varphi_1-\varphi_2\|_{\mathcal{V}}^2=\frac{C_1}{2} \|\tilde{\varphi}_1-\tilde{\varphi}_2\|_{\mathcal{V}}^2,\label{e:3.21}
\end{eqnarray}
which proves the uniform convexity of $I^{(\varepsilon)}$.
	
\textbf{Step 4.} We are ready to show the unique existence of minimizer $\tilde{\varphi}^{(\varepsilon)} \in\mathcal{V}$ of \textbf{Problem 4}, which satisfies \eqref{ff}. 	
	
Firstly, we show the continuity of $I^{(\varepsilon)}(\bar{\varphi}+\varepsilon^2\tilde{\varphi},  \bar{\varphi})$ with respect to $\tilde{\varphi}$ in $\mathcal{V}$. %$\mathcal{V}\times\{\mathcal{V}+{q_\infty e_1}\}$. 
%For the continuous of $\tilde{\varphi}$, 
Taking $\tilde{\varphi}_1$ and $\tilde{\varphi}_2$ in $\mathcal{V}$, corresponding to $\varphi_1$ and $\varphi_2$ respectively, we have 
	\begin{eqnarray}
	I^{(\varepsilon)}(\varphi_1,  \bar{\varphi})-I^{(\varepsilon)}(\varphi_2,  \bar{\varphi})&=&\varepsilon^{-4}\int_{\Omega}\left[ \frac{1}{2}\int_{|\nabla\varphi_2|^2}^{|\nabla\varphi_1|^2}\hat{\rho}^{(\varepsilon)}(\Lambda, \phi)d\Lambda -\nabla\bar{\varphi}\cdot\nabla\left(\varphi_1-\varphi_2\right)\right] dx\nonumber\\
	&=&\varepsilon^{-4}\int_{\Omega}\left[ \frac{1}{2}\int_{|\nabla\varphi_2|^2}^{|\nabla\varphi_1|^2}\hat{\rho}^{(\varepsilon)}(\Lambda, \phi)d\Lambda -\hat{\rho}^{(\varepsilon)}(|\nabla\bar{\varphi}|^2)\nabla\bar{\varphi}\cdot\nabla\left(\varphi_1-\varphi_2\right)\right] dx\nonumber\\
	&&+\varepsilon^{-2}\int_{\Omega}\left[ \varepsilon^{-2}\left(\hat{\rho}^{(\varepsilon)}(|\nabla\bar{\varphi}|^2, \phi)-1\right)\nabla\bar{\varphi}\cdot\nabla\left(\varphi_1-\varphi_2\right) \right] dx\nonumber
	\end{eqnarray} 
  
By a similar argument as done in \textbf{Step 1} to obtain \eqref{ineq-phi} and \eqref{i2}, and by the H\"older inequality, we have:
	\begin{eqnarray}
	&&\left| I^{(\varepsilon)}(\varphi_1,  \bar{\varphi})-I^{(\varepsilon)}(\varphi_2,  \bar{\varphi}) \right| \nonumber\\ 
	&\leq& \frac{\tilde{C}_1}{2}\int_\Omega|\nabla(\tilde{\varphi}_1-\tilde{\varphi}_2)|^2dx+ C\varepsilon^{-2}\int_\Omega\left(|\nabla\varphi-e_1q_{\infty}|+|\phi|\right)|\nabla(\varphi_1-\varphi_2)|dx\nonumber\\ 
	&\leq&  C||\tilde{\varphi}_1-\tilde{\varphi}_2||_{\mathcal{V}}^2
	+
	C||\tilde{\varphi}_1-\tilde{\varphi}_2||_{\mathcal{V}}. \nonumber
	\end{eqnarray} 
%	Next, we come to the continuity respect to $\bar{\varphi}$. for the same $\tilde{\varphi}$, we have  
%	\begin{eqnarray}
%	&&I^{(\varepsilon)}(\bar{\varphi}_1+\varepsilon^2\tilde{\varphi},  \bar{\varphi}_1)-I^{(\varepsilon)}(\bar{\varphi}_2+\varepsilon^2\tilde{\varphi},  \bar{\varphi}_2)\nonumber\\
%	&=&\varepsilon^{-4}\int_{\Omega}\left[ \frac{1}{2}\int_{|\nabla\bar{\varphi}_1|^2}^{|\nabla(\bar{\varphi}_1+\varepsilon^2\tilde{\varphi})|^2}\tilde{\rho}^{(\varepsilon)}(\Lambda)d\Lambda-\frac{1}{2}\int_{|\nabla\bar{\varphi}_2|^2}^{|\nabla(\bar{\varphi}_2+\varepsilon^2\tilde{\varphi})|^2}\tilde{\rho}^{(\varepsilon)}(\Lambda)d\Lambda \right] dx\nonumber\\
%	&&+\varepsilon^{-2}\int_{\Omega}\left[\nabla\tilde{\varphi}\cdot\nabla(\bar{\varphi}_1-\bar{\varphi}_2)
%	\right] dx.\nonumber
%	\end{eqnarray} 
%	By the similar calculation, we have:
%	\begin{equation}
%	\left| I^{(\varepsilon)}(\bar{\varphi}_1+\varepsilon^2\tilde{\varphi},  \bar{\varphi}_1)-I^{(\varepsilon)}(\bar{\varphi}_2+\varepsilon^2\tilde{\varphi},  \bar{\varphi}_2) \right| \leq C||\bar{\varphi}_1-\bar{\varphi}_2||_{\mathcal{V}}
%	\end{equation}
%	

	Now we can show the existence of the minimizer $\tilde{\varphi}^{(\varepsilon)}$ by the compactness argument via applying the continuity of the functional $I^{(\varepsilon)}(\bar{\varphi}+\varepsilon^2\tilde{\varphi},  \bar{\varphi})$ with respect to $\tilde{\varphi}$ in $\mathcal{V}$.

For the uniqueness, by \eqref{e:3.21}, we know that if $\varphi$ is the minimizer, then for any $\varphi_\ast\in\mathcal{V}$, we have that
\begin{equation}
I^{(\varepsilon)}(\varphi_\ast,\bar{\varphi})-I^{(\varepsilon)}(\varphi,\bar{\varphi})\geq \frac{C_1}{2}\varepsilon^{-4}\|\varphi_\ast-\varphi\|^2_{\mathcal{V}}.
\end{equation}

Therefore, 
\begin{equation}
I^{(\varepsilon)}(\varphi_\ast,\bar{\varphi})>I^{(\varepsilon)}(\varphi,\bar{\varphi}),
\end{equation}
for any $\varphi_\ast\in\mathcal{V}$ and $\varphi_\ast\neq\varphi$. It means that the minimizer is unique in $\mathcal{V}$.
	
Finally, \eqref{ff} easily follows from \eqref{I-phi} via replacing $\varphi$ by $\bar{\varphi}$, since we know that $I^{(\varepsilon)}(\varphi,\bar{\varphi})\leq I^{(\varepsilon)}(\bar{\varphi},\bar{\varphi})$.
\end{proof}

\section{Proof of Theorem \ref{MainT}}
By Remark \ref{rem:3.1}, we know that the minimizer of \textbf{Problem 4} is a weak solution of \textbf{Problem 3}.
In this section, we will first show the minimizer of \textbf{Problem 4} is a solution of \textbf{Problem 2} by remove the elliptic cut-off introduced in \textbf{Problem 3}, and then show the low Mach number limit to conclude the proof of Theorem \ref{MainT}.

\subsection{$C^{1,\alpha}$-Regularity of the minimizer}
Before removing the elliptic cut-off, let us consider the regularity of the derivatives of the solutions.
%In this section, we will show the higher regularity.

Firstly, we need the following proposition.

\begin{proposition}\label{prop-modify}
	Let $a^l_{ij}$ for $i,j = 1, \dots, n$ be $L^{\infty}$ functions on $B_1$, and $\lambda$ be a positive constant. Assume that
	$$
	\forall ~\xi\in\R^n, ~\lambda|\xi|^2\leq \sum_{i, j=1}^n a^l_{ij}\xi_i\xi_j\leq \lambda^{-1} |\xi|^2,
	~\mbox{and} ~f^l_i\in L^q, ~ q >n.
	$$
	Let $w(y)$ be a function in $H^1$. %\textcolor{red}{\sout{Let $f_i^l$ be functions in $L^q$ with $q>n$.}}
	 Suppose
	$$
	\sum_{i,j=1}^n\partial_i\left[a^l_{ij}(y) \partial_jw(y)\right] + \sum_{i=1}^n\partial_if^l_i =0
	$$
	holds in the distribution sense. Then $w(y)$ is H\"{o}lder continuous in $B_{{1}/{2}}$ and there exist two constants $0<\alpha\leq 1$, $k$, depending on $\lambda$ such that
	$$
	\sup_{y\in B_{1/2}}|w(y)|\leq k\left(||w||_{L^2(B_1)}+ ||f_i^l||_{L^q(B_1)} \right),
	$$
	$$
	\sup_{y_1, y_2\in B_{1/2}}\frac{|w(y_1)-w(y_2)|}{|y_1-y_2|^{\alpha}}
	\leq k\left(||w||_{L^2(B_1)}+ ||f_i^l||_{L^q(B_1)} \right).
	$$
\end{proposition}
The proof of this proposition can be found in \cite{Gilbarg-Trudinger} (see Theorem 8.24).

Based on Proposition \ref{prop-modify}, we can show the $C^{\alpha}$-regularity of $\nabla\tilde{\varphi}$.

\begin{lemma}\label{lem-5}
The minimizer $\tilde{\varphi}^{(\varepsilon)}$ of \textbf{Problem 4} satisfies that	%There is the continuity estimate of $\nabla\tilde{\varphi}^{(\varepsilon)}$ at infinity:
	\begin{equation}\label{5.2}
	\left|\nabla\tilde{\varphi}^{(\varepsilon)}\right|\leq
	\frac{C}{(1+|x|)^{\beta'}},
	\end{equation}
	where $\beta'=\min\{\frac{n}{2}, \beta+\frac{n}{q}-1\}$,
and that
\begin{equation}\label{4.2x}
\|\nabla\tilde{\varphi}^{(\varepsilon)}\|_{C^{\alpha}(\Omega)}\leq C,
\end{equation}
where constant $C$ is independent of $\varepsilon$. %, and $\nabla\tilde{\varphi}^{(\varepsilon)}$ is H\"{o}lder continuous.
\end{lemma}

\begin{proof} We divide the proof into three steps.
	
\textbf{Step 1.}
Let
$\Phi=\partial_k\tilde{\varphi}^{(\varepsilon)}$, $\bar{\varphi}'=\partial_k\bar{\varphi}$ for $k=1,\dots,n$. Then by the straightforward calculation, $\Phi$ satisfies
$$
\sum_{i, j =1}^n\partial_i\left(\hat{a}_{ij}\partial_j\Phi\right)+\sum_{i=1}\partial_{i}(\hat{b}_k\partial_i\varphi^{(\varepsilon)})+\varepsilon^{-2}\sum_{i, j =1}^n\partial_i\left(\hat{a}_{ij}\partial_j\bar{\varphi}'\right)=0.
$$
Since $\Delta\bar{\varphi}=0$, $\Delta\bar{\varphi}'=0$. We can change the equation above to:
$$
\sum_{i, j =1}^n\partial_i\left(\hat{a}_{ij}\partial_j\Phi\right)=-\varepsilon^{-2}\sum_{i, j =1}^n\partial_i\left((\hat{a}_{ij}-\delta_{ij})\partial_j\bar{\varphi}'\right)-\sum_{i=1}\partial_{i}(\hat{b}_k\partial_i\varphi^{(\varepsilon)}).
$$
Here, we introduce
\begin{eqnarray}
f_{ij}&=&\varepsilon^{-2}(\hat{a}_{ij}-\delta_{ij})\nonumber\\
&=&\frac{\hat{\rho}^{(\varepsilon)}-1}{\varepsilon^2}\delta_{ij}-\varepsilon^{-2}\frac{\hat{q}_\Lambda(|\nabla\varphi^{(\varepsilon)}|^2, \phi)\partial_i\varphi^{(\varepsilon)}\partial_j\varphi^{(\varepsilon)}}
{(c^\varepsilon)^2},
\end{eqnarray}
then, 
\begin{equation}
\varepsilon^{-2}\sum_{i, j =1}^n\partial_i\left((\hat{a}_{ij}-\delta_{ij})\partial_j\bar{\varphi}'\right)=\sum_{i=1}^n\partial_i(\sum_{j=1}^nf_{ij}\partial_{j}\bar{\varphi}'),
\end{equation}
Now we are going to show the uniform $L^\infty$ estimate of $f_{ij}$. 
For the first term,
\begin{eqnarray}
\frac{\hat{\rho}^{(\varepsilon)}-1}{\varepsilon^2}&=&\frac{\hat{\rho}^{(\varepsilon)}\left(|\nabla\varphi^{(\varepsilon)}|^2, \phi\right)-1}{\varepsilon^2}\nonumber\\
&=&\frac{\tilde{h}^{-1}\left(\frac{\varepsilon^2\left(q^2_\infty-\left(\hat{q}(|\nabla\varphi^{(\varepsilon)}|^2, \phi)\right)\right)}{2}+\tilde{h}(1)\right)-1}{\varepsilon^2}\nonumber\\
&=&\frac{\tilde{h}^{-1}\left(\frac{\varepsilon^2\left(q^2_\infty-\left(\hat{q}(|\nabla\varphi^{(\varepsilon)}|^2, \phi)\right)\right)}{2}+\tilde{h}(1)\right)-\tilde{h}^{-1}(h(1))}{\varepsilon^2}\nonumber\\
&=&\frac{q^2_\infty-\hat{q}(|\nabla{\varphi}^{(\varepsilon)}|^2, \phi)}{2}\int_{0}^{1} 	\left( \tilde{h}^{-1} \right)'\left(t\varepsilon^2\frac{q^2_\infty-\hat{q}(|\nabla{\varphi}^{(\varepsilon)}|^2, \phi)}{2}+\tilde{h}(1)\right) dt%\nonumber\\
%&&\quad 
\label{diffrho}
%&\leq&C
\end{eqnarray}
For the second term,
\begin{equation}
\varepsilon^{-2}\frac{\hat{q}_\Lambda(|\nabla\varphi^{(\varepsilon)}|^2, \phi)\partial_i\varphi^{(\varepsilon)}\partial_j\varphi^{(\varepsilon)}}
{(c^\varepsilon)^2}=\frac{\hat{q}_\Lambda(|\nabla\varphi^{(\varepsilon)}|^2, \phi)  \partial_i\varphi^{(\varepsilon)}\partial_j\varphi^{(\varepsilon)}}
{\tilde{p}'(\rho^\varepsilon)}.%\leq C
\end{equation}

Therefore, due to the cut-off, we have the uniform $L^{\infty}$ estimate of $f_{ij}$.

Also, for $i, k=1, \cdots, n$,  
\begin{equation}\label{4.6}
|\hat{b}_k\partial_i\varphi^{(\varepsilon)}|\leq C|\partial_k\phi|.
\end{equation}
By the assumption that $\nabla\phi\in L^q$, we know $\hat{b}_i\partial_i\varphi^{(\varepsilon)}$ are bounded in $L^q$ for $q>n$.

%On the other hand, $\partial_i\bar{\varphi}'=\partial_i\bar{\psi}'$, while $\bar{\psi}'=\partial_k\bar{\psi}$.

%For the regularity, we need to divided to two steps.

\textbf{Step 2.} Now we can show the interior estimates based on the equations derived in \textbf{Step 1}. For any bounded interior subregion of $\Omega$,  since $\bar{\varphi}'$ satisfies $\Delta\bar{\varphi}'=0$, and $||\nabla \bar{\varphi}||_{L^2}$ is uniformly bounded, from the standard elliptic estimate, we have that $\partial_j\bar{\varphi}'$ is locally $C^{\infty}$ and uniformly bounded, \textit{i.e.}, 
$$
||\partial_i\bar{\varphi}'||_{L^\infty}\leq C||\nabla \bar{\varphi}||_{L^2},
$$ 
where constant $C$ does not depend on $\varepsilon$. Then, by Proposition \ref{prop-modify}, we have the interior $L^\infty$ and local H\"{o}lder estimate of $\nabla\tilde{\varphi}^{(\varepsilon)}$.

Next for the boundary estimate near $\partial\Omega$, one can apply Theorem 8.29 in \cite{Gilbarg-Trudinger} to replace Proposition \ref{prop-modify} to follow the arguments above to show the boundary estimates near $\partial\Omega$.

\medskip
\textbf{Step 3.} Finally, let us consider the estimates of $\nabla\tilde{\varphi}$ when $|x|$ is sufficiently large. Let 
$$
\bar{\psi}':=\partial_k\bar{\psi},
$$ 
where $\bar{\psi}=\bar{\varphi}-q_{\infty}x_1$.  For any sufficiently large $R$ with $\left\{\frac{R}{2}<|x|<2R \right\} \subset \Omega$, define
\begin{equation}\label{5.5}
w(y) =R^{\beta'} \Phi(Ry)\qquad \mbox{and} \qquad v(y)=R^{\beta'}\bar{\psi}'(Ry)
\end{equation}
on $\left\{\frac{1}{2}<|y|<2\right\}$. Note that $\partial_i\bar{\varphi}'=\partial_i\bar{\psi}'$. Then on $\left\{\frac{1}{2}<|y|<2\right\}$, $w(y)$ satisfies
$$
\sum_{i, j =1}^n\frac{\partial}{\partial y_i}\left(\tilde{a}_{ij}\frac{\partial}{\partial y_j}w\right)=-\sum_{i=1}^n\frac{\partial}{\partial y_i}(\sum_{j=1}^nf_{ij}\frac{\partial}{\partial y_j}v)-\sum_{i=1}\frac{\partial}{\partial y_i}(R^{\beta'+1}\hat{b}_k\partial_i\varphi^{(\varepsilon)}).
$$

By \eqref{barpsiL2} and \eqref{ff}, we have that
%And, the uniform bound of $\phi'$ and $\bar{\psi}'$ come to 
	$$
\int_{\left\{\frac{1}{2}<|y|<2 \right\}}|w(y)|^2 dy \leq C R^{2\beta'-n}\leq C', ~ ~ \mbox{ and } ~ ~ \int_{\left\{\frac{1}{2}<|y|<2 \right\}}|v(y)|^2 dy  \leq C R^{2\beta'-n}\leq C',
$$
with $R^{\beta'+1}\tilde{b}_k\partial_{i}\varphi^{(\varepsilon)}$ being bounded by the observation that
$$
\int_{\left\{\frac{1}{2}<|y|<2 \right\}}|R^{\beta'+1}\nabla_x\phi(y)|^q dy \leq C.
$$
The inequality above is due to the second condition for $\phi$ in \eqref{psicondition1} and $\beta'+1\leq \beta+\frac{n}{q}$.

Since $v$ also satisfies $\Delta_{y} v=0$, the standard elliptic estimate implies that 
$$
\left\|\frac{\partial}{\partial y_j}v\right\|_{L^\infty(\{\frac{1}{2}<|y|<2\})}\leq C||v||_{L^2(\{\frac{1}{2}<|y|<2\})}\leq C,
$$ 
where constant $C$ only depends on the dimension.
Applying Proposition \ref{prop-modify}, we have
$$
|w(y)|\leq C.%\frac{C}{R^\frac{n}{2}}.% \quad \mbox{for} ~|y|=1.
$$
So let $|y|=1$, then
\begin{equation}
|\Phi(Ry)|=|w(y)|\leq \frac{C}{R^{\beta'}}.
\end{equation}
Going back to \eqref{5.5}, we have for sufficiently large $|x|$, and for $i=1, \cdots, n$,
$$
\left|\nabla\tilde{\varphi}^{(\varepsilon)}(x)\right|\leq \frac{C}{|x|^{\beta'}}.
$$
 
It means \eqref{5.2} holds. Now we can follow the standard argument to lift the regularity to show that $\nabla\tilde{\varphi}^{(\varepsilon)}$ is uniformly H\"older continuous (see \cite{Gilbarg-Trudinger}), \textit{i.e.}, \eqref{4.2x} holds.
%The further estimate lifting is standard. For the further detail please reference \cite{Gilbarg-Trudinger}.
\end{proof}

\subsection{Uniqueness of solutions of Problem 3}
In this subsection, we will show the uniqueness of the modified flow such that we can remove the cut-off and apply the Bers skill to show the existence of solutions of \textbf{Problem 2}. %and complete the proof of the subsonic part of main theorem. 

\begin{lemma}\label{6.1}
\textbf{Problem 3} admits a unique classical solution $\varphi$ up to a constant such that
\begin{equation}\label{unicon}
\int_{\Omega}|\nabla \varphi-(q_\infty,0,\cdots, 0)|^2 dx\leq C.
\end{equation} 
\end{lemma}
%\textbf{Proof of Theorem \ref{mainthm}}:
\begin{proof}	
Exitence is proved by Remark \ref{rem:3.1}, Theorem \ref{th3}, and Lemma \ref{lem-5}. \eqref{unicon} follows from \eqref{ff}. So we only need to consider the uniqueness.

Assume there are two different solutions of \textbf{Problem 3} $\varphi_1$ and $\varphi_2$. Let 
$$
\hat{\varphi}=\varphi_1-\varphi_2\qquad\mbox{ and }\qquad
\varphi_{\tau}=(2-\tau)\varphi_1+(\tau-1)\varphi_2.
$$ 

Since both $\varphi_1$ and $\varphi_2$ satisfy \eqref{unicon}, we have
\begin{equation}\label{hatl2}
\int_{\Omega}|\nabla \hat{\varphi}|^2 dx\leq C.
\end{equation} 

Moreover, by the straightforward calculation, $\hat{\varphi}$ satisfies that
\begin{equation}\label{diffequ}
\sum_{i,j=1}^{n}\partial_i(\check{a}_{ij}\partial_j\hat{\varphi})=0,
\end{equation}
with
\begin{equation}
\check{a}_{ij}\left(\nabla \varphi^{(\varepsilon)}\right) =\int_1^2 \hat{\rho}^{(\varepsilon)}\left(|\nabla\varphi_{\tau}|^2, \phi\right)\left(\delta_{ij} -
\frac{\hat{q}_\Lambda(|\nabla\varphi_{\tau}|^2, \phi)\partial_i\varphi_{\tau}\partial_j\varphi_{\tau}}
{(c^\varepsilon)^2}\right) d\tau.
\end{equation}

Define a series of test functions $\eta_R(x)>0$ for $R>1$ with uniform $C^2$-bounds, such the function $\eta_R$ is supported in $|x|\leq R$ and identically equals to $1$ in $|x|\leq R-1$.

By multiplying $\hat{\varphi}\eta_R$ on both the sides of \eqref{diffequ} and integrating by part in $\Omega$, we obtain:
\begin{eqnarray}\label{cutoffuni}
0&=&\int_{B_{R-1}\cap\Omega}\sum_{i,j=1}^{n}\check{a}_{ij}\partial_i\hat{\varphi}\partial_j\hat{\varphi} d x.\nonumber\\
& & +\int_{(B_{R}-B_{R-1})\cap\Omega}\sum_{i,j=1}^{n}\check{a}_{ij}\partial_i\hat{\varphi}\partial_j\hat{\varphi} \eta_R d x+\int_{(B_{R}-B_{R-1})\cap\Omega}\sum_{i,j=1}^{n}\check{a}_{ij}\partial_i\hat{\varphi}\partial_j \eta_R \hat{\varphi} d x\nonumber
\end{eqnarray}
%Since $\hat{\varphi}\in\mathcal{V}$, 
By \eqref{Hardyforv}, \eqref{5.2}, and \eqref{hatl2} , passing the limit $R\rightarrow\infty$, and by the dominant convergence theorem, the identity above becomes
\begin{equation}
\int_{\Omega}\sum_{i,j=1}^{n}\check{a}_{ij}\partial_i\hat{\varphi}\partial_j\hat{\varphi} d x=0
\end{equation}
which implies $\nabla\varphi_1=\nabla\varphi_2$ by \eqref{li-3.3}.
\end{proof}

%\begin{remark}
%	The velocity field $\nabla\varphi^{(\varepsilon)}$ depends on $\varepsilon$ continuously and in particular $\max_\Omega|\nabla\varphi^{(\varepsilon)}|$ is a continuous function of $\varepsilon$.
%\end{remark}

\begin{remark}
By Remark \ref{rem:3.1} and Lemma \ref{6.1}, it is easy to see when $|\nabla\varphi^{(\varepsilon)}|<\mathring{q}^{\varepsilon_0}_\theta$, $\varphi^{(\varepsilon)}$ is equal to the solution obtained in \cite{Dong2}.
\end{remark}

\subsection{Proof of Theorem \ref{MainT}}
%\textbf{Proof of Theorem \ref{MainT}}:	

In this subsection, we will conclude the proof of Theorem \ref{MainT} by showing the solutions of \textbf{Problem 3} are solutions of \textbf{Problem 2}, and then consider the convergence rate of the low Mach number limit.
 
\begin{proof}
%Firstly, we will prove the part (1): the subsonic case.

Up to now, we have shown that for any given fixed cut-off parameters $\theta$ and $\varepsilon_0$, there exists a unique solution of \textbf{Problem 3}, which is denoted as $\varphi^{(\varepsilon)}(x;\varepsilon_0, \theta)$. 
%To remove the cut-off, which is introduced in Section 3, we define the quantity:
%$$
%\mathcal{M}^{\varepsilon}(\bar{\varphi}, \theta)=\max_{x\in\Omega}\left(\frac{|\nabla\varphi^{(\varepsilon)}(x; \bar{\varphi}, \theta)|}{c^\varepsilon}\right).
%$$
It is noticeable that for a given $\theta\in(0,1)$, if $|\nabla\varphi^{(\varepsilon)}(x;\varepsilon_0, \theta)|<\mathring{q}^{\varepsilon_0}_\theta(\phi)$, then $\varphi^{(\varepsilon)}(x;\varepsilon_0, \theta)$ is the unique solution of \textbf{Problem 3}. Note that
\begin{equation}\label{4.13}
|\nabla\varphi^{(\varepsilon)}(x;\varepsilon_0, \theta)|= |\varepsilon^2\nabla\tilde{\varphi}^{(\varepsilon)}(x;\varepsilon_0, \theta)+\nabla\bar{\varphi}(x)|\leq \max|\nabla\bar{\varphi}|+C(\varepsilon_0, \theta) \varepsilon^2.
\end{equation}
Then, there exists $\varepsilon_{0, \theta}\leq \varepsilon_0$ such that  $|\nabla\varphi^{(\varepsilon)}(x;\varepsilon_0, \theta)|<\mathring{q}^{\varepsilon_0}_\theta(\phi)$,  for any $0<\varepsilon<\varepsilon_{0, \theta}$. From the definition and uniqueness of $\varphi^{(\varepsilon)}$, $\{\varepsilon_{0, \theta}\}$ is a non-decreasing sequence respect to $\theta$ with upper bound $\varepsilon_0$. Then, we introduce $\varepsilon_{0, cr}=\varlimsup_{0<\theta<1}\varepsilon_{0, \theta}$ such that for $0<\varepsilon<\varepsilon_{0, cr}$, there exists a unique solution  $\tilde{\varphi}(x; \varepsilon_0)$, such that
\begin{equation}
|\nabla \varphi^{(\varepsilon)}(x;\varepsilon_0)|=|\varepsilon^2\nabla\tilde{\varphi}^{(\varepsilon)}(x;\varepsilon_0)+\nabla\bar{\varphi}(x)|< \mathring{q}^{\varepsilon_0}_{cr}(\phi),
\end{equation} 
which means $M^\varepsilon(\phi)<1$.  In this case, the cut-off can be removed such that the solution is a solution of \textbf{Problem 2}.

After removing the subsonic cut-off, we will optimize the critical value $\varepsilon_{cr}$. 
For each $0<\varepsilon_0<1$,  there exists an $\varepsilon_{0, cr}$, with $0<\varepsilon_{0, cr}\leq \varepsilon_0<1$. Then, the critical value $\varepsilon_c=\sup_{0<\varepsilon<1}\varepsilon_{0, cr}$ satisfies that for any $\varepsilon\in(0, \varepsilon_c)$, $0<M^\varepsilon(\phi)<1$ and $|\nabla\tilde{\varphi}^{(\varepsilon)}|$ is uniform bounded with respect to $\varepsilon$.

Finally, let us consider the convergence rate of the low Mach number limit. Note that $\varphi^{(\varepsilon)}=\bar{\varphi}+\varepsilon^2\tilde{\varphi}^{(\varepsilon)}$ holds in $C^{1,\alpha}(\Omega)$, so $\nabla\varphi^{(\varepsilon)}=\nabla\bar{\varphi}+\nabla\tilde{\varphi}^{(\varepsilon)}$ in the H\"{o}lder space, which equals to
\begin{equation}\label{uconergencerate}
u^{\varepsilon}=\bar{u}+\varepsilon^2\tilde{u}^{(\varepsilon)}.
\end{equation}
It is noticeable that $\phi\in W^{1, q}_{loc}$ for $q>n$, so $\phi$ is in some H\"{o}lder space. 
Therefore, for the density, by \eqref{2.21}, $\rho^\varepsilon\in C^\alpha(\Omega)$. Then $p^\varepsilon\in C^\alpha(\Omega)$. By the straightforward computation like the one in \eqref{diffrho}, we have  
\begin{equation}\label{rhoconergencerate}
\rho^{\varepsilon}=1+O(\varepsilon^2).
\end{equation} 
%which is the same from the analysis before. 
Consequently, the definition of the Mach number yields $M^\varepsilon=O(\varepsilon)$.
Finally, the gradients of the pressure satisfy
\begin{eqnarray}
\nabla p^\varepsilon-\nabla \bar{p}&=& -\mbox{div}(\rho^\varepsilon u^\varepsilon\otimes u^\varepsilon)+\mbox{div}(\bar{u}\otimes \bar{u})\nonumber\\
&=&\mbox{div}( \bar{u}\otimes \bar{u}-\rho^\varepsilon u^\varepsilon\otimes u^\varepsilon ).
\end{eqnarray}
From \eqref{uconergencerate} and \eqref{rhoconergencerate}, we can conclude: in the weak sense, 
\begin{equation}
\nabla p^\varepsilon=\nabla \bar{p}+O(\varepsilon^2).
\end{equation}

It completes the proof of Theorem \ref{MainT}.
\end{proof}

\bigskip
\textbf{Ackowledgments}: The authors would like to thank Professor Song Jiang for evaluable suggestions. The research of Mingjie Li is supported by the NSFC Grant No. 11671412. The research of Tian-Yi Wang  was supported in part by the NSFC Grant No. 11601401 and the Fundamental Research Funds for the Central Universities(WUT: 2017 IVA 072 and 2017 IVB 066). The research of Wei Xiang was supported in part by the Grants Council of the HKSAR, China (Project No. CityU 21305215, Project No. CityU 11332916, Project No. CityU 11304817 and Project No. CityU 11303518).


\begin{thebibliography}{99}
	
	\bibitem{alazard1} T. Alazard,  Incompressible limit of the nonisentropic Euler equations with the solid wall boundary conditions. \textit{Advances in Differential Equations} \textbf{10(1)} (2005) 19--44.
	
%	\bibitem{alazard2}T. Alazard,  Low Mach number limit of the full Navier-Stokes equations. \textit{Arch. Rational Mech. Anal.} \textbf{180(1)} (2006) 1--73.
	
%		\bibitem{Asano}
%	\newblock K. Asano,
%	\newblock {On the incompressible limit of the compressible Euler equations,}
%	\newblock \textit{Japan J. Appl. Math.}, \textbf{4} (1987) 455-488.
%	
%	\bibitem{Bers1} L. Bers, An existence theorem in two-dimensional gas dynamics, \textit{Proc. Symposia Appl. Math.} \textbf{1} (1949) 41--46.
	
%	\bibitem{Bers2} L. Bers, Boundary value problems for minimal surfaces with singularities at infinity, \textit{Trans. Amer. Math. Soc.} \textbf{70} (1951) 465--491.
	
	\bibitem{Bers3} L. Bers,  Existence and uniqueness of a subsonic flow past a given profile. \textit{Comm. Pure Appl. Math.} \textbf{7} (1954) 441--504.


%\bibitem{DBDL}
%\newblock D. Bresch, B. Desjardins, E. Grenier and C.-K. Lin, 
%\newblock {Low Mach number limit of viscous polytropic flows: formal asymptotics in the periodic case,}
%\newblock \textit{Stud. Appl. Math.}, \textbf{109} (2002) 125-149.

	
%		\bibitem{BDG} D. Bresch, B. Desjardins and E. Grenier, 
%	\newblock \emph{Oscillatory limit with changing eigen values: a formal study}, in: A. V. Fursikov, G. P. Galdi and V. V. Pukhnachev (Eds.),``New Directions in Mathematical Fluid Mechanics", 
%	\newblock Birkh\"{a}user, Basel, 2010, pp.91-105.
	
%	\bibitem{Keriss1} G. Browning, A. Kasahara, and H. Kreiss, Initialization of the primitive equations by the bounded derivative method. \textit{Journal of the Atmospheric Sciences}, \textbf{37(7)} (1980) 1424--1436.
	
	\bibitem{CCZ} G.-Q. Chen, C. Christoforou, and Y. Zhang, Continuous dependence of entropy solutions to the Euler equations on the adiabatic exponent and Mach number. \textit{Arch. Rational Mech. Anal.} {\bf 189(1)} (2008), 97--130.
	
	\bibitem{CHWX} G.-Q. Chen, F.-M. Huang, T.-Y. Wang, and W. Xiang, Incompressible Limit of Solutions of Multidimensional Steady Compressible Euler Equations, \textit{Z. Angew. Math. Phys.} (2016) 67-75 .
	
	\bibitem{CHWX1} G.-Q. Chen, F.-M. Huang, T.-Y. Wang, and W. Xiang,
	Steady Euler flows with Large Vorticity and Characteristic Discontinuities in Arbitrary Infinitely Long Nozzles, \textit{Preprint at arXiv:1712.08605} 2018.
	
	\bibitem{CDX} J. Cheng, L. Du, and W. Xiang, Incompressible R\'{e}thy Flows in Two Dimensions, \text{SIAM J. Math. Anal.}, 49 (2017), 3427--3475.
	
%	\bibitem{Courant-Friedrichs} R. Courant and K.O. Friedrichs, {\it Supersonic
%		Flow and Shock Waves}, Interscience Publishers Inc.: New York, 1948.

%	\bibitem{CC}%(MR1148892)
%	\newblock S. Chapman and T. G. Cowling,
%	\newblock ``The Mathematical Theory of Non-Uniform Gases",
%	\newblock 3rd edition, Cambridge University Press, 1990.

%\bibitem{Danchin}
%\newblock R. Danchin,
%\newblock {Low Mach number limit for viscous compressible flows,} 
%\newblock \textit{ESAIM: Mathematical Modelling and Numerical Analysis}, \textbf{39} (2005) 459-475.

\bibitem{DWX}
\newblock X. Deng, T.-Y. Wang and W. Xiang,
\newblock {Three-dimensional full Euler flows with nontrivial swirl in axisymmetric nozzles,}
\newblock \text{SIAM J. Math. Anal.} \textbf{50}, 2740--2772.
	
%		\bibitem{DeG}
%	\newblock B. Desjardins and E. Grenier,
%	\newblock {Low Mach number limit of viscous compressible flows in the whole space,}  
%	\newblock \textit{Proc. R. Soc. Lond. Ser. A Math. Phys. Eng. Sci.}, \textbf{455} (1999) 2271-2279.
	
%	\bibitem{DGLM}
%	\newblock B. Desjardins, E. Grenier, P.-L. Lions and N. Masmoudi,
%	\newblock  {Incompressible limit for solutions of the isentropic Navier-Stokes equations with Dirichlet boundary conditions,} 
%	\newblock \textit{J. Math. Pures Appl.}, \textbf{78} (1999) 461-471.
	
	
	
	
%	\bibitem{Dong1} G.-C. Dong, \textit{Nonlinear partial differential equations of second
%		order}, Translations of Mathematical Monographs. American
%	Mathematical Society, Providence, RI, 1991.
	
	
	\bibitem{Dong2} G.-C. Dong, and B. Ou. Subsonic flows around a body in space, \textit{Comm. Partial Differential Equations}, \textbf{18(1-2)} (1993) 355--379.
	
%		\bibitem{Dou-Jiang-Ou}
%	\newblock C. Dou, S. Jiang and Y. Ou
%	\newblock {Low Mach number limit of full Navier-Stokes equations in a 3D bounded domain}
%	\newblock \textit{J. Differential Equations}, \textbf{258} (2015) 379-398.
	
	\bibitem{Edin} D. Ebin, The motion of slightly compressible fluids viewed as a motion with strong constraining force. \textit{Annals of mathematics}, (1977) 141--200.
	
	\bibitem{ve}V. Elling. Nonexistence of low-mach irrotational inviscid flows around polygons. \textit{J. Diff. Eqns.}, 262 (2017) 2705--2721.
	
%		\bibitem{Fan-Gao-Guo}
%	\newblock J. Fan, H. Gao and B. Guo,
%	\newblock {Low Mach number limit of the compressible magnetohydrodynamic equations with zero thermal conductivity coefficient,}
%	\newblock \textit{Math. Methods Appl. Sci.} \textbf{34} (2011) 2181-2188.
	
	\bibitem{FX}
	B. Fang and W. Xiang,
	\newblock The uniqueness of transonic shocks in supersonic flow past a 2-D wedge,
	\textit{J. Math. Anal. Appl.} \textbf{437}, (2016) 194--213. 
	
	\bibitem{Finn1} R. Finn, and D. Gilbarg, Asymptotic behavior and uniqueness of plane subsonic flows, \textit{Comm. Pure Appl. Math.} \textbf{10} (1957), 23--63.
	
	\bibitem{Gilbarg1} R. Finn, and D. Gilbarg,  Three-dimensional subsonic flows and asymptotic estimates for elliptic partial
	differential equations, \textit{Acta Math.} \textbf{98} (1957) 265--296.
	
	
		\bibitem{Feireisl-N} 
	\newblock E. Feireisl, and A. Novotny,
	\newblock \textit{Singular Limits in Thermodynamics of Viscous Fluids},
	\newblock Birkh\"{a}user, Basel, 2009.
	
	\bibitem{Frankl} F. Frankl, and M.  Keldysh, Die \"aussere neumann'she aufgabe f\"ur nichtlineare elliptische differentialgleichungen mit anwendung auf die theorie der flugel im kompressiblen gas. \textit{Bull. Acad. Sci.} \textbf{12} (1934) 561--697.
	
	
		\bibitem{Gilbarg-Trudinger}
	\newblock D. Gilbarg, and N. Trudinger,
	\newblock \textit{Elliptic partial differential equations of second order},
	\newblock Springer-Verlag, New York,1983, second edition.
	
	\bibitem{Gu-Wang1} X. Gu, and T.-Y. Wang, On subsonic and subsonic-sonic flows in the infinity long nozzle with general conservatives force, \textit{Acta Mathematica Scientia} 
	\textbf{37} (2017) 752--767.
	
	\bibitem{Gu-Wang2} X. Gu, and T.-Y. Wang, On subsonic and subsonic-sonic flows with general conservatives force in exterior domains, Accepted by \textit{Acta Mathematicae Applicatae Sinica}

	
%		\bibitem{Hu-Wang}
%	\newblock X.-P. Hu and D.-H. Wang,
%	\newblock {Low Mach number limit of viscous compressible magnetohydrodynamic flows,}
%	\newblock \textit{SIAM J. Math. Anal.}, \textbf{41} (2009) 1272-1294.
	
	
	
%		\bibitem{Iguchi}
%	\newblock T. Iguchi,
%	\newblock {The incompressible limit and the initial layer of the compressible Euler equation in $\mathbb{R}_+^n$,}
%	\newblock \textit{Math. Methods Appl. Sci.}, \textbf{20} (1997) 945-958.
	
	\bibitem{Isozaki1}
	\newblock H. Isozaki,
	\newblock {Singular limits for the compressible Euler equation in an exterior domain,}
	\newblock \textit{J. Reine Angew. Math.}, \textbf{381} (1987) 1-36.
	
%	\bibitem{Isozaki2}
%	\newblock H. Isozaki, 
%	\newblock {Singular limits for the compressible Euler equation in an exterior domain, I. Bodies in an uniform flow,} \newblock \textit{Osaka J. Math.}, \textbf{26} (1989) 399-410.
	
%		\bibitem{Jiang-Ou} 
%	\newblock S. Jiang and Y.B. Ou, \newblock {Incompressible limit of the non-isentropic Navier–Stokes equations with well-prepared initial data in three-dimensional bounded domains,}
%	\newblock \textit{J. Math. Pures Appl.}, \textbf{96} (2011) 1-28.
%	
%	\bibitem{JJL1}
%	\newblock S. Jiang, Q.C. Ju and F.C. Li,
%	\newblock {Incompressible limit of the compressible magnetohydrodynamic equations with vanishing viscosity coefficients,}
%	\newblock \textit{SIAM J. Math. Anal.}, \textbf{42} (2010) 2539-2553.
%	
%	
%	\bibitem{JJL2}
%	\newblock S. Jiang, Q.C. Ju and F.C. Li,
%	\newblock {Low Mach number limit for the multi-dimensional full magnetohydrodynamic equations,}
%	\newblock \textit{Nonlinearity} \textbf{15} (2012) 1351-1365.
%	
	
	\bibitem{JJL3}
	\newblock S. Jiang, Q. Ju and F. Li,
	\newblock {Incompressible limit of the non-isentropic ideal magnetohydrodynamic equations,}
	\newblock \textit{SIAM J. Math. Anal.} \textbf{48} (2016), no. 1, 302-319. 
	
	\bibitem{JJLX} S. Jiang, Q. Ju, F. Li, and Z.-P. Xin, Low Mach number limit for the full compressible magnetohydrodynamic equations with general initial data. \textit{Advances in Mathematics} \textbf{259} (2014) 384--420.
	
%	\bibitem{Kim}
%\newblock H. Kim and J. Lee, 
%{The incompressible limits of viscous polytropic fluids with zero thermal conductivity coefficient,}
%\newblock \textit{Comm. Partial Differential Equations}, \textbf{30} (2005) 1169-1189.
	
	\bibitem{KM1} S. Klainerman and A. Majda, Singular perturbations of quasilinear hyperbolic systems with
	large parameters and the incompressible limit of compressible fluids, \textit{Comm. Pure Appl. Math.} \textbf{34} (1981) 481--524.
	
	\bibitem{KM2} S. Klainerman and A. Majda, Compressible and incompressible fluids, \textit{Comm. Pure Appl. Math.} \textbf{35} (1982), 629--653.
	
%	\bibitem{Keriss2} H. Kreiss, Problems with different time scales for partial differential equations. \textit{Comm. Pure Appl. Math.} \textbf{33} (1980) 399-�C439.



	
	
	
%		\bibitem{Levermore}
%	\newblock C. D. Levermore, W. Sun and K. Trivisa,
%	\newblock {A low Mach number limit of a dispersive Navier-Stokes system,}
%	\newblock \textit{SIAM J. Math. Anal.}, \textbf{44} (2012) 176-1807.
	
		\bibitem{Li-Wang-Xiang} M. Li, T.-Y. Wang, and W. Xiang,
	\textit{Low Mach Number Limit of Steady Euler Flows in Multi-Dimensional Nozzles} Preprint.
	
	\bibitem{Lions P.-L.01} P.-L. Lions and N. Masmoudi, Incompressible limit for a viscous compressible fluid, \textit{J. Math. Pures Appl.} \textbf{77} (1998) 585--627.
	
%	\bibitem{Lions-Masmoudi}P.-L. Lions and N. Masmoudi, On a free boundary barotropic model, \textit{Ann. Inst. H. Poincar$\acute{e}$ Anal. Non Lin$\acute{e}$aire},  \textbf{16} (1999) 373--410.
	
		\bibitem{ou2} G. Lu, and B. Ou, A Poincare Inequality on $R^n$ and Its Application to Potential Fluid Flows in Space. \textit{Comm. Appl. Nonlinear Anal.} \textbf{12(1)} (2005) 1--24.
	
		\bibitem{Masmoudi}
	\newblock N. Masmoudi,
	\newblock \textit{Incompressible, inviscid limit of the compressible Navier-Stokes system,} 
	\newblock \textit{Ann. Inst. H. Poincar\'{e} Anal. Non Lin\'{e}aire}, \textbf{18} (2) (2001) 199-224.
	
%	\bibitem{Masmoudi1} N. Masmoudi, Asymptotic problems and compressible-incompressible limit, In {\it Advances in mathematical fluid mechanics}, pp. 119--158, Springer-Verlag, 2000.
	
	\bibitem{Masmoudi2} N. Masmoudi, Examples of singular limits in hydrodynamics, {\it Handbook of differential equations: evolutionary equations}, {\bf 3}, pp. 195--275, 2007.
	
	\bibitem{MS} G. M\'{e}tivier, and S. Schochet,  The incompressible limit of the non-isentropic Euler equations. \textit{Arch. for Rational Mech. Anal.} \textbf{158(1)} (2001) 61--90.
	
%		\bibitem{MS1}
%	\newblock G. M\'{e}tivier and S. Schochet, 
%	\newblock {Averaging theorems for conservative systems and the weakly compressible Euler equations,}
%	\newblock \textit{J. Differential Equations}, \textbf{187} (2003) 106-183.
	
	\bibitem{ou1} B. Ou, An irrotational and incompressible flow around a body in space. \textit{J. of PDEs} \textbf{7(2)} (1994) 160--170.
	
	\bibitem{Payne}L.  Payne and H.  Weinberger, Note on a lemma of Finn and Gilbarg. \textit{Acta Math.} \textbf{98} (1957) 297--299.
	
	\bibitem{QX} A. Qu and W. Xiang, Three-dimensional steady supersonic Euler flow past a concave cornered wedge with lower pressure at the downstream. \textit{Arch. for Rational Meth. Anal.} \textbf{228} (2018) 431--476. 
	
	\bibitem{Schiffer} M. Schiffer, Analytical theory of subsonic and supersonic flows in \textit{Handbuch der Physik.}, pp. 1--161, Springer Berlin Heidelberg, 1960.
	
%		\bibitem{Schochet1}
%	\newblock S. Schochet,
%	\newblock {The compressible Euler equations in a bounded domain: existence of solutions and the incompressible limit,}
%	\newblock \textit{Comm. Math. Phys.}, \textbf{10} 4 (1986) 49-75.
	
	
	\bibitem{Schochet} S. Schochet, The mathematical theory of low Mach number flows. \textit{ESAIM: Mathematical Modelling and Numerical Analysis} \textbf{39(03)} (2005) 441--458.
	
%	\bibitem{Shiffman1} M. Shiffman, On the existence of subsonic flows of a compressible fluid, \textit{Proc. Nat. Acad. Sci. U.S.A.} \textbf{38}	(1952) 434--438.
	
	\bibitem{Shiffman2} M. Shiffman, On the existence of subsonic flows of a compressible fluid, \textit{J. Rational Mech. Anal.} \textbf{1} (1952) 605--652.
	
	\bibitem{Ukai} S. Ukai,	The incompressible limit and the initial layer of the compressible Euler equation, \textit{J. Math. Kyoto Univ.} \textbf{26} (1986) 323--331.
	
	\bibitem{XZZ} W. Xiang, Y. Zhang, and Q. Zhao, Two-dimensional steady supersonic exothermically reacting Euler flows with strong contact discontinuity over Lipachitz wall, \textit{Preprint at arXiv:1709.03263}, 2017.
	
	\bibitem{VD} M. Van Dyke, \textit{Perturbation methods in fluid mechanics} (Vol. 964), New York: Academic Press, 1964.
	
\end{thebibliography}
\end{document}